\documentclass[11pt,letterpaper]{amsart}
\usepackage{amsmath, amsthm, amssymb}
\usepackage[colorlinks=true, pdfstartview=FitV, linkcolor=blue,citecolor=blue, urlcolor=blue]{hyperref}
\usepackage{mathtools}
\usepackage[T1]{fontenc}
\usepackage[utf8]{inputenc}

\usepackage{graphicx}
\usepackage{subcaption}

\usepackage{pdfsync}

\usepackage{parskip} \newtheorem{proposition}{Proposition}[section]
\newtheorem{theorem}[proposition]{Theorem}
\newtheorem{corollary}[proposition]{Corollary}
\newtheorem{lemma}[proposition]{Lemma}
\theoremstyle{definition}
\newtheorem{definition}[proposition]{Definition}
\newtheorem{remark}[proposition]{Remark}

\numberwithin{equation}{section}

 \def\R{\mathbb{R}}
  \def\Z{\mathbb{Z}}
  
  \def\C{\mathbb{C}}

  \def\ve{\varepsilon}

\newcommand{\ee}{\mathrm{e}}
\newcommand{\p}{\partial}
\newcommand{\jj}{J}
\newcommand{\na}{\nabla }
\renewcommand{\d}{\mathrm{d}}
\renewcommand{\div}{\mathrm{div}}
\newcommand{\re}{\mathrm{Re}}
\newcommand{\im}{\mathrm{Im}}
\newcommand{\res}{\mathrm{res}}
\newcommand{\eqnb}{\begin{equation}}
\newcommand{\eqne}{\end{equation}}
\newcommand{\loc}{\mathrm{loc}}

\renewcommand*{\thefootnote}{(\arabic{footnote})}

\newcommand\blfootnote[1]{%
  \begingroup
  \renewcommand\thefootnote{}\footnote{#1}%
  \addtocounter{footnote}{-1}%
  \endgroup
}

\title[Logarithmic spiral vortex sheets]{Well-posedness of logarithmic spiral vortex sheets}

\author{Tomasz Cie\'{s}lak, Piotr Kokocki and Wojciech S. O\.za\'nski}

\begin{document}

\begin{abstract}
We consider a family of 2D logarithmic spiral vortex sheets which include the celebrated spirals introduced by Prandtl \emph{(Vortr\"age aus dem Gebiete der Hydro- und Aerodynamik, 1922)} and by Alexander \emph{(Phys. Fluids, 1971)}. 
We prove that for each such spiral the normal component of the velocity field remains continuous across the spiral. We give sufficient conditions for spiral vortex sheets to be weak solutions of the 2D incompressible Euler equations. Namely, we show that a spiral gives rise to such a solution if and only if two conditions hold across every spiral: a \emph{velocity matching} condition and a \emph{pressure matching} condition. Furthermore we show that these two conditions are equivalent to the imaginary part and the real part, respectively, of a single complex constraint on the coefficients of the spirals. This in particular provides a rigorous mathematical framework for logarithmic spirals, an issue that has remained open since their introduction by Prandtl in 1922. Another consequence of the main result is well-posedness of the symmetric Alexander spiral with two branches, despite recent evidence for the contrary. 
Moreover, our result implies a sharpness result of Delort's theorem on global existence of solutions to the Euler equations with initial vorticity measure.
Our main tools are new explicit formulas for the velocity field and for the pressure function, as well as a notion of a \emph{winding number} of a spiral, which not only gives a robust way of localizing the spirals' arms with respect to a given point in the plane, but also ensures correct asymptotic behaviour near $0$.
\end{abstract}



\maketitle
\blfootnote{T.~Cie\'{s}lak: Institute of Mathematics, Polish Academy of Sciences, \'Sniadeckich 8, 00-656 Warszawa, Poland, email: cieslak@impan.pl  

P.~Kokocki: Faculty of Mathematics and Computer Science, Nicolaus Copernicus University, Chopina 12/18, 87-100 Toru\'n, Poland, email: pkokocki@mat.umk.pl

W. S. O\.za\'nski: Department of Mathematics, University of Southern California, Los Angeles, CA 90089, USA, and Institute of Mathematics, Polish Academy of Sciences, \'Sniadeckich 8, 00-656 Warszawa, Poland, email: ozanski@usc.edu }
\section{Introduction}
We are concerned with the two-dimensional incompressible Euler equations,
\eqnb\label{euler_intro}
\begin{cases}
&\p_t v + v\cdot\nabla v + \nabla p =0, \\
&\div \, v =0,
\end{cases}
\eqne
on $\R^2 \times (0,\infty )$, where $v$ denotes the velocity of an inviscid fluid and $p$ denotes the pressure function.
By considering the vorticity $
\omega \coloneqq \p_1 v_2 - \p_2 v_1$ we can rewrite the Euler equations \eqref{euler_intro} in the vorticity form 
\eqnb\label{euler_vort}
\p_t \omega + v \cdot \nabla\omega =0,
\eqne
where $v$ is recovered from $\omega$ by the Biot-Savart law,
\eqnb\label{bs_law}
v = K \ast \omega ,\qquad \text{ where }\qquad K(x)\coloneqq  \frac{1}{2\pi } \frac{x^\perp}{|x|^2}, \quad (x_1,x_2)^\perp \coloneqq (-x_2, x_1).
\eqne

One can also consider a \emph{weak solution of \eqref{euler_intro}} as any 
weakly divergence-free vector field $v\in L^2_{\loc} (\R^2 \times (0,\infty ))$ such that 
\eqnb\label{weak_euler}
\int_0^\infty \int_{\R^2} \left( v \cdot \p_t \varphi + \sum_{1\le i,j\le 2}v_i v_j \p_i \varphi_j \right) =0
\eqne
for all divergence-free $\varphi \in C_0^\infty (\R^2 \times (0, \infty );\R^{2})$.

The main purpose of this work (see Theorem~\ref{thm_main} below) is to provide a characterization of solutions of the $2$D incompressible Euler equations in the form of self-similar vortex sheets. To this end we introduce
\begin{definition}[Vortex sheet velocity fields]\label{def_velocity}
We say that $v\colon \R^2 \to \R^2$ is a \emph{vortex sheet velocity field} if
\eqnb\label{def_of_v} v (x,y,t) =  t^{\mu -1} w \left( \frac{x}{t^\mu }, \frac{y}{t^\mu } \right)
\eqne
for some $w \in L^2_{loc} (\R^2 )$ such that 
\begin{enumerate}
\item $\Omega \coloneqq \mathrm{curl}\, w = \p_1 w_2 - \p_2 w_1$ is supported on a curve $\Sigma  \subset \R^2$ that is locally smooth on $\R^2 \setminus \{ 0 \}$,
\item $w\cdot n$ is continuous across $\Sigma$, where $n(z) $ is a normal vector to $\Sigma $,
\item $w$ is a potential flow away from $\Sigma$ in the sense that $w^* (z)= \Phi'(z)$ for $z\in \C \setminus (\Sigma \cup \{ 0 \}) $, where $\Phi$ is holomorphic for such $z$, and $w$ is continuous up to $\Sigma$ from either side,
\item $\int_{\p B(0,\eta )} (|w|^2 + |q| ) \d s \to 0 $ as $\eta \to 0$, where $q$ denotes the self-similar pressure profile using the Bernoulli's law in the complex form,
\eqnb\label{Bernoulli_complex_ss}
q(z) \coloneqq - \re \left( (2\mu -1) \Phi(z) - \mu z w^*(z) \right) - \frac{1}2 |w(z)|^2 .
\eqne
\end{enumerate}
\end{definition}
 Here we used the complex representation $z=x+iy=r\ee^{i\theta }\in\C\setminus(\Sigma\cup\{0\})$ of a point $(x,y)\in \R^2 $, the complex representation $w= w_1+iw_2$ of a velocity field $(w_1,w_2)\in \R^2$, and we denoted the complex conjugate of $z\in \C$ by $z^*$. Note also that, given self-similar profiles $\Omega$, $q$, $\Phi$, we have 
\[
\omega(z,t) =  t^{-1 } \Omega \left( \frac{z}{t^\mu } \right),\qquad p(z,t) =  t^{2\mu-2 } q\left( \frac{z}{t^\mu } \right),\qquad \Psi(z,t) =  t^{2\mu-1 } \Phi \left( \frac{z}{t^\mu } \right),
\] 
where $\Psi$ denotes the potential function, and the Bernoulli law \eqref{Bernoulli_complex_ss} takes the form 
\eqnb\label{def_of_p}
p(z,t) \coloneqq - \p_t\re\, \Psi (z,t) - \frac12 |v(z,t)|^2.
\eqne

Our main result characterizes vortex sheet velocity fields as weak solutions of the Euler equations \eqref{weak_euler}.

\begin{theorem}[Main result]\label{thm_main}
A vortex sheet velocity field $v$ is a weak solution of the Euler equations \eqref{euler_intro} if and only if the \emph{velocity matching} condition,
\eqnb\label{velocity_matching}
n(z,t) \cdot \left( v(z,t) t - \mu z \right) =0, 
\eqne
 and the \emph{pressure matching} condition,
\eqnb\label{pressure_matching}
p (z,t)\text{ is continuous at }z,
\eqne
hold for all $z\in \Sigma (t)$, $t>0$, where $n(z,t) $ is a normal vector to $\Sigma (t)$ at $z\in \Sigma (t)$, and $p$ is given by the Bernoulli law \eqref{def_of_p}.
\end{theorem}
We note that in the velocity matching condition \eqref{velocity_matching} we consider $n\cdot v $ with $v$ defined as the limit velocity field from either side of $\Sigma$, which is well-defined by the assumption of the continuity of $v\cdot n$ across $\Sigma$. 

While conditions (2), (3) of Definition~\ref{def_velocity}, as well as the pressure matching condition \eqref{pressure_matching} are well-known in the vortex sheet literature, condition (4) of Definition~\ref{def_velocity} and the velocity matching \eqref{velocity_matching} seem to appear for the first time. We emphasize that  these two conditions are crucial for the spiral vortex sheet to solve the 2D Euler in a weak sense. In fact, they are sharp in the sense which we describe below.

A recent work \cite{bcs} considers the limit $t \to 0^+$ of Kaden's spiral, which was identified in \cite[Proposition~4.5]{COPS}. It is a steady velocity field whose vorticity  is supported on the line $\{(x,0)\in \R^2, x\geq 0\}$ with density of the vorticity given by 
\[
\gamma(x)=(2-1/\mu)x^{2-1/\mu},\;\;\mu\in [2/3,1).       
\]
In the case $\mu=2/3$ it is shown that the velocity field satisfies conditions (1)--(3) of Definition~\ref{def_velocity}, as well as both the velocity matching condition \eqref{velocity_matching} and the pressure matching \eqref{pressure_matching}. However, the decay condition (4) of Definition~\ref{def_velocity} fails, and it is observed that the Euler equations \eqref{weak_euler} fail in a neighbourhood of $0$, see \cite{bcs} for details. 

On the other hand, if $\mu\in (2/3,1)$, then all conditions (1)--(4) of Definition~\ref{def_velocity} hold, and so does the pressure matching \eqref{pressure_matching}. However, the velocity matching \eqref{velocity_matching} fails, and so the Euler equations \eqref{weak_euler} are again not satisfied. 
This shows that both the decay condition (Definition~\ref{def_velocity}(4)) and the velocity matching \eqref{velocity_matching} are relevant.

Secondly we note that the pseudovelocity  $w-\mu z$ plays an essential role in the self-similar flows \eqref{def_of_v}. In fact, writing the vorticity equation \eqref{euler_vort} in the self-similar form gives 
\[
\Omega (z) = (w (z) - \mu z ) \cdot \nabla \Omega (z),
\]
see \cite{bressan_murray}, which shows that $\Omega$ grows exponentially along the characteristic lines of pseudovelocity. The velocity matching condition \eqref{velocity_matching} states that the pseudovelocity has to be tangent to the spiral for the vortex sheet to satisfy the 2D Euler equations \eqref{weak_euler}. 

Moreover, it appears that the two matching conditions \eqref{velocity_matching}--\eqref{pressure_matching} play equally important role, see Theorem~\ref{thm_alex} for details.

Furthermore, we note that the lack of integrability of the Biot-Savart law \eqref{bs_law} is one of the main challenges of mathematical description of spiral vortex sheets. This is particularly clear in the case of logarithmic spiral vortex sheets, which we are concerned with in this paper.

We define logarithmic spirals using the parametrization of vorticity:
\eqnb\label{log_spirals}
\begin{cases}
Z_m(\theta , t) = t^\mu \ee^{a (\theta - \theta_m )} \ee^{i\theta },\\
\Gamma_m (\theta ,t) = g_m t^{2\mu -1 }  \ee^{2a (\theta - \theta_m )}, \\
r_m(\theta ,t)  = | Z_m (\theta ,t)| = t^\mu \ee^{a (\theta - \theta_m )}, \qquad \theta \in \R, \ t>0,
\end{cases}
\eqne
where $\Gamma_m (\theta , t)$ denotes the total circulation at radius $r_m(\theta , t)$, and $Z_m(\theta , t)$ denotes the parametrization of the spiral at time $t$ with respect to the parameter $\theta $, see Fig.~\ref{fig_spiral_intro} for a sketch. Here
\eqnb\label{log_spirals_parameters}
a>0, \quad \mu \in \R , \quad g_m \in \R\setminus\{0\}, \quad \theta_m \in [0,2\pi ),
\eqne
where $m=0,\ldots , M-1$, are parameters and it is assumed (without loss of generality) that
\[
0\leq \theta_0 < \theta_1 < \ldots < \theta_{M-1} ,
\]
for simplicity. In such case we will denote by $\Sigma_m (t)$ the spiral parametrized by $Z_m (\theta , t )$, and we set $\Sigma_m \coloneqq \Sigma_m (1)$, $\Sigma \coloneqq \bigcup_{m=0}^{M-1} \Sigma_m$ and $Z_m (\theta )\coloneqq Z_m(\theta ,1)$ for brevity. For each $m$ we denote by $\Omega_m (t)$ the region between $\Sigma_m (t)$ and $\Sigma_{m+1} (t)$ (with the convention that $\Sigma_{M} \coloneqq  \Sigma_0$), and similarly $\Omega_m \coloneqq \Omega_m (1)$.

\begin{figure}[h]
\centering
 \includegraphics[width=0.3\textwidth]{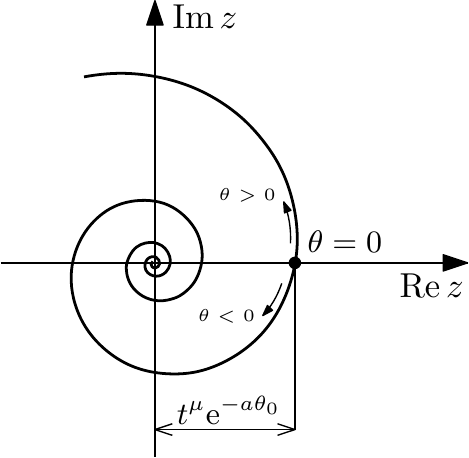}
 \nopagebreak
  \captionof{figure}{The sketch of a single ($M=1$) logarithmic spiral.}\label{fig_spiral_intro} 
\end{figure} 

We will show that one can describe the above vortex sheets using the vortex sheet velocity \eqref{def_of_v} with  the self-similar profile
\begin{equation}\label{def_of_w}
w(z) \coloneqq  \ee^{i\theta } \sum_{k=0}^{M-1} \frac{2ag_{k} }{r(a-i)} \left(  r^{\frac{2a}{a+i} }\ee^{A  (\theta_k -\theta ) } \frac{\ee^{2\pi \jj (r,\theta ,k)A }}{1-\ee^{2\pi A}} \right)^*,
\end{equation}
where $z=r\ee^{i\theta }\in \C \setminus (\Sigma \cup \{ 0 \})$ and
\eqnb\label{def_of_jj}
\jj =\jj (r,\theta ,k ) \coloneqq \min \left\lbrace j\in \Z \colon a(2\pi j + \theta_k - \theta ) + \ln r >0 \right\rbrace
\eqne
denotes the \emph{winding number of the spiral}. Here $z=x+iy=r\ee^{i\theta }\in\C\setminus(\Sigma\cup\{0\})$ denotes the complex representation of a point $(x,y)\in \R^2 \setminus(\Sigma\cup\{0\})$, and $v=v_1+iv_2$ is a complex representation of the velocity field $(v_1,v_2)\in \R^2$. We note that here $\theta \in \R$, and we show (in Proposition~\ref{prop_j1}(i)) that the definition \eqref{def_of_jj} ensures that $w$ is well-defined; namely that \eqref{def_of_w} is invariant with respect to adding $2\pi k$, $k\in \Z$, to $\theta$. The formula \eqref{def_of_w} has a number of remarkable properties which are valid for all values of the parameters \eqref{log_spirals_parameters}, including conditions (1)--(4) of Definition~\ref{def_velocity}, as well as decay properties, which we discuss in detail in Theorem~\ref{thm1} below. 

We first note that applying Theorem~\ref{thm_main} to \eqref{def_of_w} gives the first characterization of the logarithmic spirals \eqref{log_spirals} as solutions to the Euler equations \eqref{weak_euler}. 

 \begin{theorem}[Well-definiteness of the logarithmic spiral vortex sheets]\label{thm_alex}
In the case of $w$ given by \eqref{def_of_w}, the  
velocity matching condition \eqref{velocity_matching} and the pressure matching condition \eqref{pressure_matching} are equivalent to the imaginary part and the real part of 
 \begin{equation}\label{eq-disc2}
\frac{1}{\sinh (\pi A) }\sum_{k=0}^{M-1} \mathcal{A}_{mk}g_{k} = \frac{2i}{A} \left(\mu  + \frac{(1-2\mu)(a+i)}{2a}\right)^{*} = -(a^2+1-2\mu +2a\mu i )/2a^2
\end{equation}
for $m\in \{ 0 , \ldots , M-1 \}$, where
\begin{equation}\label{def_of_A}
A\coloneqq -\frac{2ai}{a+i} = \frac{-2a}{1+a^2} \left( 1+ai \right),
\end{equation}
and
\eqnb\label{def_of_Amk} 
\mathcal{A}_{mk}\coloneqq  \ee^{A(\theta_{k}-\theta_{m})}
\begin{cases}
\ee^{-\pi A} \hspace{2cm}&  k>m ,\\
\cosh(\pi A) & k=m , \\
\ee^{\pi A}  & {k}<{m}
\end{cases}
\eqne
for $k,m\in \{ 0, \ldots , M-1 \}$. 
 \end{theorem} 
 
 We note that \eqref{eq-disc2} implies in particular that the left-hand side must be independent of $m$. 
 Theorem~\ref{thm_alex} says that, in the case of logarithmic spirals, the Euler equations \eqref{weak_euler} reduce to single algebraic equation in $\C^M$, which lets one determine the coefficients \eqref{log_spirals_parameters}. It also shows that, as mentioned above, the velocity matching \eqref{velocity_matching} and pressure matching \eqref{pressure_matching} play, roughly speaking, equally relevant roles as the  imaginary and real parts of \eqref{eq-disc2}. We note that equation \eqref{eq-disc2} has appeared in the work of Elling and Gnann \cite{EL} using the viewpoint of the Birkhoff-Rott equations in which case must be supplemented with certain compatibility conditions (see \eqref{cc-1}). One of the main conclusion of this paper is that the compatibility conditions can be neglected, and we discuss this issue in detail in Section~\ref{sec_EL} below.

The problem of the lack of integrability of the Biot-Savart law \eqref{bs_law} for logarithmic spirals, discussed above, has been attempted in a number of ways. Saffman \cite{saffman}  introduced a way of dealing with this problem that is based on an unpublished note of Moore, see Section~8.3 in his extensive exposition \cite{saffman} on vortex dynamics.  The idea presented in \cite{saffman} is based on an observation that the integral in the Biot-Savart law \eqref{Biot_Sav} would converge if ``$2a$'' appearing in the exponent in the vorticity parametrization \eqref{log_spirals} was replaced by any $\beta \in (0,a)$ (see the discussion following \eqref{spirals_eq} below for a more extensive analysis of this issue). It then suggests that the velocity field for $\beta =2a$ could be obtained by analytic continuation. However, such approach does not guarantee that the Biot-Savart law \eqref{bs_law} continues to hold for the analytic continuation. Moreover, this method also leads to incorrect decay properties as $r\to 0$, which is a subtlety only visible through the explicit velocity formula \eqref{def_of_w}. To be more precise, according to \cite[eq. (9) in Section~8.5]{saffman} the method of analytic continuation predicts that $|w(z)|\leq C |z|^{(a^2-1)/(a^2+1)}$ for each $z\in \C$, while in fact the correct estimate is $|w(z)|\leq C|z|$, see \eqref{growth_of_w} below. This is a problem, as this suggests that the decay condition (4) in Definition~\ref{def_velocity} fails for $a\leq \sqrt{3}/3$, which suggests ill-posedness as in the Kaden example mentioned above. However this is false - the condition (4) holds in fact for all $a>0$! The reason for this discrepancy in the decay rate is because this method is  only concerned with an attempt to extend the the Biot-Savart law \eqref{bs_law} into the non-integrable regime, rather than attempt to explain the issue from using fundamental principles. In particular it neglects some important contributions to the growth from the dynamics on the spirals which is encoded by the winding number $J(r,\theta , k)$ introduced above. We discuss this issue in detail in Section~\ref{sec_prop_of_v}.

We also note that Wu \cite[p.~6]{wu} has proposed a method that makes sense of the Birkhoff-Rott equations \eqref{br_eqs_intro} if the integral on the right-hand side of \eqref{br_eqs_intro} converges for at least one $\theta \in \R$. However, this does not simplify the problem in the case of logarithmic spiral \eqref{log_spirals}, as the integral carries the same difficulty for every $\theta \in \R $ due to the self-similarity of the spiral.\\

We emphasize that Theorem~\ref{thm_alex} provides a unified well-definiteness result of all logarithmic spirals, including  Alexander spirals \cite{alexander}, namely the symmetric choice 
\eqnb\label{def_of_alex}
g_m\coloneqq  g \in \R\setminus \{ 0\}, \quad \theta_m \coloneqq \frac{2\pi m}M.
\eqne
In such case \eqref{eq-disc2} becomes
\eqnb\label{constraint_for_alex}
a^2+1-2\mu + 2a\mu i =-2a^2 g \coth (\pi A/M),
\eqne
and so we obtain the following.
 \begin{corollary}[Well-definiteness of the Alexander spirals]\label{cor_alex}
Alexander's spirals \eqref{def_of_alex} give rise to weak solutions of the Euler equations if and only if $a>0$, $g\in \R\setminus \{ 0 \}$, $\mu \in \R$ satisfy \eqref{constraint_for_alex}.
 \end{corollary} 

We note that there exist infinitely many triples $(a,g,\mu)$ satisfying \eqref{constraint_for_alex}. For example, first choosing any $a>0$ such that the direction of the term on the right-hand side of \eqref{constraint_for_alex} is different from the direction of $-1+ai$, we see that there exists a unique pair of $\mu \in \R$ and $g \in \R \setminus \{ 0\}$ for which \eqref{constraint_for_alex} holds. In particular, such choice of $a$ can be achieved simply by taking it sufficiently large. Indeed, noting that $A= -2i - 2/(a+i)$ we have $\ee^{\pi A } = \ee^{-2\pi/(a+i)}$, and so
\[
\coth (\pi A) = \frac{\ee^{-2\pi/(a+i)}+\ee^{2\pi/(a+i)}}{\ee^{-2\pi/(a+i)}-\ee^{2\pi/(a+i)}} = \frac{2+O\left( \frac{1}{a^2} \right) }{\frac{-4\pi}{a+i } + O\left( \frac{1}{a^2} \right) } = \frac{-a}{2\pi } + O(1)
\]
as $a\to \infty$, which shows that $\coth (\pi A)$ and $(-1+ai)$ have different directions for sufficiently large $a$. Thus the unique choice of $\mu, g $ follows by expressing $1+a^2$ as a linear combination of the two vectors. Note that $g\ne 0$ (as required), as $1+a^2$ and $-1+ai$ are linearly independent.\\

We also note that the particular case of Alexander spirals with $M=1$ is often referred to as the \emph{Prandtl spiral} \cite{prandtl22}, and  Corollary~\ref{cor_alex} provides the first rigorous mathematical framework of such spiral, an issue that has remained open since its introduction in 1922 in the pioneering work of Prandtl \cite{prandtl22}. We also note that the  cases $M=1$ and $M=2$ in Corollary~\ref{cor_alex} are particularly interesting, as in these cases some recent evidence suggests ill-posedness of the spirals, see Section~\ref{sec_EL} for details. \\

Furthermore, the well-posedness of the Alexander spirals in Corollary~\ref{cor_alex} can be used to prove sharpness of Delort's theorem on global-in-time well-posedness of the Euler equations with $\omega_0\in H^{-1} (\R^2)$. 

\begin{theorem}[Th\'eor\`eme~1.1.1 in Delort~\cite{delort}]\label{thm_delort}
Let $\omega_0\in H^{-1} (\R^2)$ be a compactly supported, nonnegative  Radon measure on $\R^2$, and let $v_0\coloneqq K\ast \omega_0$. Then there exists   $v\in L^\infty_{loc} (\R ; L^2_{loc} (\R^2 ; \R^2))$ and   $p\in L^\infty_{loc} (\R; \mathcal{S}'(\R^2 ))$ satisfying the Euler equations \eqref{euler_intro} with $v(0)=v_0$.
\end{theorem}

In fact, the Prandtl spiral can be used to show that, if the assumption on compact support is relaxed to $\sigma$-finite measures and $H^{-1}$ is replaced by $H^{-1}_{loc}$, the above theorem either fails or produces nonunique solutions.

\begin{corollary}[Sharpness of Delort's Theorem~\ref{thm_delort}]\label{cor_sharpness}
There exists a nonnegative $\sigma$-finite Radon measure $\omega_0\in H^{-1}_{loc} (\R^2)$ and a solution $v$ with $v(0)$ such that $\mathrm{curl}\,v(0)=\omega_0$ in the sense of distributions and  $\int_{B(0,1)} |v(t)|^2 \to \infty$ as $t\to t_0$, for some $t_0>0$. 
\end{corollary}

\begin{proof}
See Section~\ref{sec_sharpness}.
\end{proof}

We emphasize that, until recently, the Alexander spirals \eqref{def_of_alex} were the only known example of logarithmic spiral vortex sheets. However, Theorem~\ref{thm_alex} gives a characterization \eqref{eq-disc2} of all logarithmic spirals. We note that a recent result \cite{cko1} of the authors constructs a generic family of nonsymmetric spirals with angles $\theta_m$'s approximately equal to halves of the Alexander angles \eqref{def_of_alex}, which gives another application of Theorem~\ref{thm_alex}. However, existence of other spirals exhibit some lack of symmetry, is an important open problem. In this context, we also refer to  \cite[Section~5]{EL}, who have suggested existence of some nontrivial family of such spirals. 

We note that the issue of existence of global-in-time solutions of the Euler equations in the form of vortex sheets with no distinguished sign remains an important open problem.  So far, besides some trivial stationary solutions (e.g.~piecewise constant shear flows), very few such objects are known, which remains an interesting problem with some contributions by \cite{Lopes_x2_Zhouping} in a  mirror-symmetric setting and by \cite[Section~2.2]{elgindi_jeong}. In this context Theorem~\ref{thm_alex} simplifies the problem into a finite dimensional equation \eqref{eq-disc2} in the logarithmic spiral case. We also note that a recent result of Mengual and Sz\'ekelyhidi Jr.~\cite{mengual} demonstrates a construction of solutions with vortex sheet initial data, using the method of   convex integration.

\subsection{Historical background}\label{sec_history}

Logarithmic spirals \eqref{log_spirals} were first introduced in the case $M=1$ in 1922 by Ludwig Prandtl \cite{prandtl22} as a model of vortices detaching from tips of wings, and were used to study the influence of the vortices on the lift of the wings as well as the influence of the vortices behind a given plane on a plane following it; see Fig.~4 in~\cite{prandtl05} and Fig.~2 and 4 in~\cite{prandtl22}. We also refer the reader to the early works of Helmholtz \cite{helmholtz_87,helmholtz_68}, which include some of the first remarks on vortex sheets. Prandtl's contributions to the study of the spirals, long time before an introduction of the notion of weak solutions of the Euler equations or the Birkhoff-Rott \cite{birkhoff,rott} equations, were based on studying putative discontinuities of the velocity field, dimensional analysis as well as on physical experiments.

 In fact, Prandtl himself has commented (on p.~704 in \cite{prandtl22}) that his exposition has ``touched on many unfinished things'' and that its purpose was to ``give suggestions for further research, rather than to report final results''. 

The multi-branched spirals  \eqref{log_spirals} were introduced by Alexander \cite{alexander}, who observed that the velocity flow generated by the spirals is divergence-free and potential, which made it possible to use the method of a complex potential. The question of mathematical well-posedness, however, remained open. 

The introduction of Birkhoff-Rott \cite{birkhoff,rott}  equation,
\eqnb\label{br_eqs_intro}
\partial_{t}Z(\Gamma  ,t) = \left(\frac{1}{2\pi i}\mathrm{p.v.}\int \frac{\d \Gamma'}{Z(\Gamma  ,t)-Z(\Gamma',t)}\right)^{*},
\eqne
promised a more rigorous approach to Prandtl's spirals. Here $Z(\Gamma ,t)\in \C$ denotes the position of a point on the curve that is the support of vorticity, $\Gamma $ is the cumulative vorticity. 

The Birkhoff-Rott equations are based on a generalization of the Biot-Savart law \eqref{bs_law} to the case of vorticities $\omega_{t}$ in the form of measures, that is
\eqnb\label{Biot_Sav}
v(x,t) = \frac{1}{2\pi}\int_{\R^{2}}\frac{(x-y)^{\bot}}{|x-y|^{2}}\,\d \omega_{t}(y).
\eqne
 Note that we do not apply the principal value in the integral in \eqref{Biot_Sav} as it suffices to define $v$ only outside of $\Sigma (t)$ (recall the weak formulation \eqref{weak_euler} of the Euler equations), while $\omega$ is supported on $\Sigma (t)$. As for the definition of $v$ on $\Sigma (t)$ one could employ the principal value (as in \eqref{br_eqs_intro}) or, equivalently, apply the common convention of considering the arithmetic mean of the limit velocities from both sides of the sheet (see \cite[(2.15)--(2.16)]{Lopes}). We will show below (in Remark~\ref{rem_average}) that this assumption arises naturally in our approach to vortex sheets.  

  Thus the Birkhoff-Rott equation \eqref{br_eqs_intro} describes  the evolution of the support of the sheet by using only its parametrization  $Z(\theta  ,t)$, and the vorticity distribution can be described by 
\[
\omega_{t} = \gamma (t) \delta_{\Sigma (t)},
\]
where $\Sigma (t) \coloneqq \{ Z(\theta  ,t ) \colon \theta  \in \R \}$ is the interface curve, and $\gamma (Z(\theta,t),t) $ is the difference of the tangential parts of limit velocities across $Z(t)$.

 Unfortunately, the right-hand side of the Birkhoff-Rott equation \eqref{br_eqs_intro} does not converge at large scales in the case of logarithmic spirals \eqref{log_spirals} (see \cite{kambe} for instance), and thus any analysis of the Prandtl spiral remained out of reach, including the theory of vortex sheets as weak solutions of Euler equations, developed by Di Perna and Majda \cite{DiPM}. In fact, the divergence of the right-hand side of the Birkhoff-Rott equation \eqref{br_eqs_intro} is connected to the problem of defining the velocity field $v$ via the Biot-Savart law \eqref{Biot_Sav}, due to large contributions of vorticity from the large scales. To be more precise note that expressing $\Gamma $ in \eqref{log_spirals} in terms of $r$ (by substituting $\ee^{a(\theta - \theta_0 )}$) gives that
\eqnb\label{vort_growth_prandtl}
\omega(B(0,r))=gt^{-1} r^2.
\eqne
Thus, if we fix $t>0$ and we ``cut'' the spiral at some distance from the origin then, not only the velocity field is well-defined, but also $\omega\in H^{-1}_{\loc}$ (see \cite{CS}), which, due to compact support, implies that $v\in L_{\loc}^2$, due to a result of Schochet \cite{schochet}. Without cutting the spiral at large scales it has recently been proved in \cite[Theorem 2.1]{COPS} that the Biot-Savart law \eqref{Biot_Sav} defines an $L^1_{\loc}$ velocity field $v$ if the vorticity $\omega$ is a nonnegative $\sigma$-finite measure on $\R^2$ with
\eqnb\label{cops_restriction}
\int_{\R^2} \frac{\d \omega (x)}{1+|x|} <\infty,
\eqne
which does not hold for Prandtl spirals due to \eqref{vort_growth_prandtl} above.

Another approach of understanding vortex sheets is based on the weak vorticity formulation of the 2D Euler equations, i.e.,
\[
\int_0^\infty \int_{\R^2} \left(  \p_t \phi (x,t)  + \int_{\R^2} H_\phi (x,y,t) \d \omega_t (y)  \right) \d \omega_t (x) \,\d t + \int_{\R^2 } \phi(x,0) \d \omega_0 (x) =0
\]
for all $\phi \in C_0^\infty (\R^2 \times (0,\infty ))$, where $\omega_0$ denotes the initial vorticity and
\[
H_{\phi } (x,y,t) \coloneqq \frac{\na \phi (x,t)- \na \phi (y,t)}{2} \cdot K(x,y)
\]
and $K$ is the Biot-Savart kernel \eqref{bs_law}.

 Note that this formulation does not involve the velocity function itself, but rather is concerned with the vorticity $\omega $ only. This formulation has been introduced by \cite{schochet2} (and appears implicitly already in \cite{delort}) in order to deal with the lack of continuity of $v$ across a sheet, and \cite{Lopes} showed equivalence of the weak vorticity formulation with the the Birkhoff-Rott equation \eqref{br_eqs_intro}. However, the relation with the weak formulation \eqref{weak_euler} is still unclear. An equivalence between the two formulations was established by Delort \cite{delort} for vorticities that remain bounded in time with values in the space of bounded Radon measures and with values in $H^{-1}_{\loc}$ and such that the initial vorticity $\omega_0$ has compact support, recall  Theorem~\ref{thm_delort} and Corollary~\ref{cor_sharpness} above. \\
 
One more approach to understanding the logarithmic spirals, is to replace the single Prandtl spiral with multiple coexisting logarithmic spirals that are obtained by rotating a single spiral around the origin, namely expect that the case $M>2$ in \eqref{log_spirals} gives some cancellation of contributions from large scales which makes this case easier to analyse than the case $M=1$. 
This has been explored by  Elling and Gnann \cite{EL} in the context of Birkhoff-Rott equations and offers a remarkable observation that captures such cancellations.

\subsection{The work of Elling and Gnann}\label{sec_EL}
In their remarkable work, Elling and Gnann \cite{EL} applied the ansatz \eqref{log_spirals} in \eqref{br_eqs_intro} to obtain  
\eqnb\label{spirals_eq}
\mu  + \frac{(1-2\mu)(a+i)}{2a}=
\left(\frac{1}{2\pi i}\mathrm{p.v.}\int_{-\infty}^{\infty} \sum_{k=0}^{M-1}\frac{2ag_k \ee^{2a \sigma }\,\d\sigma  }{1-\ee^{(a+i) \sigma + i (\theta_k-\theta_m )} }\right)^{*}
\eqne
where $\sigma\coloneqq \theta' - \theta +\theta_m-\theta_k$. It is clear from \eqref{spirals_eq} that the right-hand side, while integrable for $\sigma \to- \infty$ for each pair $k,m\in \{ 0, \ldots , M-1\}$, is not integrable as $\sigma \to \infty$, as expected from the above discussion of the Prandtl spiral. However, they observed the algebraic identity
\eqnb\label{intro_alg_id}
\frac{\ee^{2a\sigma}}{1-\ee^{(a+i)\sigma+i\Delta }} = \frac{\ee^{-2i\sigma-2i\Delta}}{1-\ee^{(a+i)\sigma+i\Delta}} - \ee^{-2i\sigma-2i\Delta} - \ee^{(a-i)\sigma-i\Delta},
\eqne
for  $\sigma , \Delta \in \R$. By taking $\Delta \coloneqq \theta_k - \theta_m$ this gives that   \eqnb\label{same_integrand_intro}
\sum_{k=0}^{M-1}
\frac{g_k \ee^{2a\sigma}}{1-\ee^{(a+i)\sigma+i(\theta_k - \theta_m )}} = \sum_{k=0}^{M-1} \frac{g_k \ee^{-2i\sigma-2i(\theta_k - \theta_m )}}{1-\ee^{(a+i)\sigma+i(\theta_k - \theta_m )}} \qquad \text{ for all }\sigma \in \R
\eqne
provided the compatibility conditions
\begin{align}\label{cc-1}
\sum_{k=0}^{M-1} g_{k}\ee^{-i\theta_{k}} = 0,\qquad \sum_{k=0}^{M-1} g_{k}\ee^{-2i\theta_{k}} = 0
\end{align}
hold. Remarkably, assuming the compatibility conditions we can apply \eqref{same_integrand_intro} to substitute the integrand in \eqref{spirals_eq} for sufficiently large $\sigma>0$ and hence obtain the missing integrability at $\sigma\to\infty$. This illustrates the expected cancellation of multiple logarithmic spirals mentioned above. Under the compatibility conditions \eqref{cc-1}, the equation \eqref{spirals_eq} is equivalent to \eqref{eq-disc2},
\[ \frac{1}{\sinh (\pi A) }\sum_{k=0}^{M-1} \mathcal{A}_{mk}g_{k} = \frac{2i}{A} \left(\mu  + \frac{(1-2\mu)(a+i)}{2a}\right)^{*} = -(a^2+1-2\mu +2a\mu i )/2a^2 . \]

Thus the problem of finding logarithmic spirals satisfying the Birkhoff-Rott equations reduces to finding the parameters \eqref{log_spirals_parameters} that satisfy the compatibility conditions \eqref{cc-1} and the constraint \eqref{eq-disc2}. We emphasize that the main difference between this and Theorem~\ref{thm_alex} is that the latter does not require compatibility conditions \eqref{cc-1}.

This is a big improvement, which can already be seen in the case of Alexander spirals \eqref{def_of_alex}, in which case the compatibility conditions \eqref{cc-1} hold only for $M\geq 3$. In fact, in such case $z_k \coloneqq \ee^{-i\theta_k}$, $k=0,\ldots , M-1$, are roots of the equation $z^M-1=0$, and \eqref{cc-1} are simply the first two  Vi\'ete formulas of this polynomial equation. 

On the other hand, the compatibility conditions \eqref{cc-1} fail for $M=2$  for the same reason since $\ee^{-i\theta_0}$ and $\ee^{-i \theta_1}$ are solutions to the quadratic equation $z^2-1=0$, and so Vi\'ete formulas give
\eqnb\label{contradiction_for_M2}
\ee^{-2i\theta_{0}}+\ee^{-2i\theta_{1}}= \left( \ee^{-i\theta_{0}}+\ee^{-i\theta_{1}}\right)^2 -2\ee^{-i\theta_0}\ee^{-i\theta_1} =2, 
\eqne
which contradicts \eqref{cc-1}. The case $M=1$ is similar.\\

Thus, although the Birkhoff-Rott approach, developed by Elling and Gnann \cite{EL}, gives satisfactory description of Alexander spirals for $M\geq 3$, it also gives the misleading suggestion that both the Prandtl spiral ($M=1$), as well as the Alexander spirals with $M=2$, cannot be given a rigorous mathematical meaning. This demonstrates that the approach of the Euler equations \eqref{weak_euler}, provided by Theorem~\ref{thm_alex}, should be used. \vspace{0.3cm}\\

\subsection{Birkhoff-Rott and Euler equivalence}\label{sec_intro_br_vs_euler}
On the other hand, if the compatibility conditions \eqref{cc-1} do hold, we also show that a solution to the Birkhoff-Rott equation corresponds to the same solution, but in terms of the velocity field, to the Euler equations. 

\begin{theorem}\label{thm4}
Suppose that the compatibility conditions \eqref{cc-1} hold. Then, outside the arms of the spirals $\Sigma(t)$ the velocity field $v(t)$ can be obtained by the integral Biot-Savart law \eqref{bs_law} applied to measure \eqref{log_spirals}, that is, 
\begin{align*}
v(z,t) =  \left(\frac{1}{2\pi i}\int_{-\infty}^{\infty}\sum_{k=0}^{M-1}\frac{2at^{2\mu-1}g_{k}\ee^{2a \sigma }\,d\sigma }{z-t^{\mu}\ee^{(a+i) \sigma +i \theta_k }}\right)^{*}
\end{align*}
for $z\not\in\Sigma(t)\cup \{ 0\}$, $t>0$. 
\end{theorem}

Theorem~\ref{thm4} can be proved using methods of contour integration, which gives further insight into the right-hand side of \eqref{def_of_w}, particularly the meaning of the winding number $\jj$. 

We note that ${v}$ from Theorem~\ref{thm4} equals, on each spiral, the average of its limit values from the left and right sides of the spiral\footnote{In fact, our definition \eqref{def_of_w} of $w$ could have included this detail. We did not include this for brevity, as well as due to the fact that the weak formulation \eqref{weak_euler} is only concerned with ``almost everywhere'' definition of $w$. In fact, the only property of the behaviour of $w$ on the spirals themselves that we use is the velocity matching condition \eqref{velocity_matching}, which is concerned with the normal component of $w$ that remains continuous anyway, due to Theorem~\ref{thm1}(ii) below. }, which often appears as an assumption of models of vortex sheets flow, see \cite{birkhoff,MR1867882,marpul,saffman}. Indeed, this is a consequence of the Biot-Savart \eqref{bs_law}, since points on the sheet correspond to the contour crossing through a single pole, in which case contour integration gives the average of the integrals over the contours that avoid the pole from each side, see Remark~\ref{rem_average} for details. This in particular answers the open question posed by \cite{Lopes} on p.~4132, for such spirals.

\subsection{Properties of the velocity formula (1.12)}\label{sec_prop_of_v}

Finally, we discuss some remarkable properties of the velocity formula \eqref{def_of_w}. In fact, the formula is a vortex sheet velocity field in the sense of Definition~\ref{def_velocity} for \emph{any values of the parameters} \eqref{log_spirals_parameters}, including the choices not satisfying the Euler equations \eqref{eq-disc2}. Moreover, the formula implies continuity of the velocity field from either side of the vortex sheet $\Sigma (t)$, as well as provides precise asymptotic estimates of $v$ and $p$ at $r\to 0$ and at $r\to \infty$, thanks to the winding number $J$. \\

In order to see this, we first observe that
\[
w^* (r\ee^{i\theta } ) = iA (r\ee^{i\theta })^{iA-1} \underbrace{\sum_{k=0}^{M-1} \frac{g_k \ee^{A(\theta_k +2\pi \jj (r,\theta , k)) }}{1-\ee^{2\pi A}} }_{=: D(r,\theta )}
\]
which shows that
\eqnb\label{w_simple_form}
w^* (z) = D\,iA z^{iA-1}.
\eqne
Thus, $w^*$ is not only holomorphic in $\C\setminus (\Sigma \cup \{ 0 \} )$, i.e. between the spirals, but also  the flow is described by the complex potential, that is
\[
w^* = \Phi',
\]
where
\eqnb\label{def_of_potential_fcn_Phi}
\Phi (r \ee^{i\theta }) \coloneqq  D\, (r\ee^{i\theta })^{iA} = r^{iA } \sum_{k=0}^{M-1} g_k \ee^{A(\theta_k - \theta )}\frac{\ee^{2\pi \jj (r,\theta , k)A}}{1-\ee^{2\pi A}}.
\eqne

This observation has a number of consequences. First, we note that \eqref{w_simple_form} suggests that $|w (z)| \sim |z|^{\re\,(iA-1)} = |z|^{(a^2-1)/(a^2+1)}$ for $z$ close to zero, which would indicate singular behaviour of $v$ at the origin for  $a\in (0,1/3]$, as discussed below \eqref{def_of_Amk} in relation to Saffman's \cite[eq. (9) in Section~8.5]{saffman} calculation. We show (in Section~\ref{sec_pf_thm1}) that this suggestion is false, as there is additional dependence on $r$ hidden in $D$ via the winding number $\jj (r,\theta ,k)$. Note that $\jj$ remains constant only when $r$ and $\theta $ are coupled; in particular fixing $\theta \geq 0$ and taking $r\to 0$ the definition \eqref{def_of_jj} gives that $\jj (r,\theta , k)\to +\infty$. We show below (in Section~\ref{sec_pf_thm1}) that this implies that in fact
\eqnb\label{growth_of_w}
|w (z)| \leq C(\theta_k,g_k,a) |z|,\qquad z\in \C,
\eqne
and so in particular $v\in L^2_{\loc}(\R^2 \times (0, \infty ))$ and $v$ has locally finite kinetic energy, $\| v(t) \|_{L^2 (B(x_0,1))} <\infty $ for $x_0\in \R^2$ and each $t>0$. Moreover,  
\[ |\Phi (z)| = |z/A|^{-1} |w(z)| \leq C(\theta_k,g_k,a) |z|^2, \qquad z\in \C.\]

Using  Bernoulli's law \eqref{def_of_p} we can determine the pressure function $p(z,t)$ explicitly and obtain its estimates. This, as well as other properties of $v$, are summarized in the following. 
\begin{theorem}\label{thm1}
Let $a>0$, $\mu\in\R$, $\theta_m \geq 0$, $g_m\in \R$ for $ m \in \{ 0, \ldots ,  M-1 \}$. If $v$ is defined by \eqref{def_of_v} then 
\begin{enumerate}
\item[(i)] The profile function $w^{*}$ is a holomorphic map on the set $\C\setminus(\Sigma\cup\{0\})$, satisfies the pointwise estimate \eqref{growth_of_w} and admits continuous extension to the closure of the region $\Omega_{m}$ for any $0\le m\le M-1$.
\item[(ii)] Given $z\in\Sigma(t)$, let $v^{R}(z,t)$ and $v^{L}(z,t)$ denote the limit velocities $v(z',t)$ as  $z'$ tends to $z$ from the right and left sides of $\Sigma(t)$, respectively. Then 
\begin{equation}
n(z,t)\cdot (v^{R}(z,t) - v^{L}(z,t))=0,\quad t>0, \ z\in \Sigma(t),
\end{equation}
where $n(z,t)$ denotes the unit normal vector to $\Sigma(t)$ at $z$. In particular $v$ is weakly divergence-free.
\item[(iii)] The vorticity $\mathrm{curl}\,v$ is a locally finite measure of the form of logarithmic spirals \eqref{log_spirals}, i.e.,
\begin{equation}\label{rotation}
\mathrm{curl}\, v(t) = \sum_{k=0}^{M-1}\gamma (t)\delta_{\Sigma_{k}(t)}, \quad t>0,
\end{equation}
in the sense of distributions, where $\gamma(Z_{k}(\theta,t),t)\coloneqq {\p_\theta \Gamma_{k}(t,\theta)}\left|{\p_\theta  Z_{k}(\theta,t)}\right|^{-1}$.
\item[(iv)] $v$ is a classical solution of the Euler equations in the region $\{ (x,t) \colon t>0, x\in \Omega_m (t)\}$ between the spirals, with the pressure function $p(z,t)$ defined by the Bernoulli law \eqref{def_of_p}, and
\eqnb\label{growth_p}
| p(z,t)| \leq C (\theta_k , g_k, a) t^{-1} |z|^2
\eqne
\end{enumerate}
\end{theorem}
We emphasize that the claims of the above theorem hold for \emph{all values} of the parameters \eqref{log_spirals_parameters}. 

The structure of the article is as follows. In the following Section~\ref{sec_prelim} we recall some basic concepts related to a description of 2D fluid flows in the complex plane, as well as an elementary technical lemma. In Section~\ref{sec_vel_field} we first discuss the notion of the winding index \eqref{def_of_jj} (Section~\ref{sec_winding_number}) and we prove Theorem~\ref{thm1} (Section~\ref{sec_pf_thm1}). In the following Section~\ref{sec_euler} we prove our main result, Theorem~\ref{thm_main}, and in Section~\ref{sec_log_solutions} we prove the well-definiteness of the logarithmic spiral vortex sheets as solution to the Euler equations \eqref{weak_euler}, Theorem~\ref{thm_alex}, including the equivalence of the velocity matching condition \eqref{velocity_matching} and the pressure matching condition \eqref{pressure_matching} to the imaginary and real parts of \eqref{eq-disc2} in Sections~\ref{sec_vel_matching} and \ref{sec_p_matching}, respectively. We prove the sharpness of Delort's Theorem~\ref{thm_delort}, Corollary~\ref{cor_sharpness}, in Section~\ref{sec_sharpness}, and the equivalence of the Birkhoff-Rott approach and the Euler approach, Theorem~\ref{thm4}, in Section~\ref{sec_bs}.

\section{Preliminaries}\label{sec_prelim}

We denote the complement of $A\subset \R^2$ by $A^c$, and we denote the complex conjugate of $z=x+iy$ by $z^*=x-iy$.
We say that a holomorphic function
\[
\Psi  = \psi  + i \chi 
\]
is the complex potential function of $v^*=v_1-iv_2$ if
\[
v^* = \Psi',
\]
where $f'$ denotes the  derivative of a complex function $f$. In such case
\[
\begin{split} v_1 (x,y) &= \psi_x (x,y) = \chi_y (x,y),\\
v_2(x,y) &= \psi_y (x,y) = -\chi_x (x,y).
\end{split}
\]
In other words, in real variables, the velocity field $(v_1,v_2)$ satisfies 
\[
(v_1,v_2)= \na \psi = -\na^\perp \chi,
\]
where $\na^\perp \coloneqq ( -\p_y , \p_x)$. Note that existence of the complex potential implies that the flow $(v_1,v_2)$ is irrotational, due to the Cauchy-Riemann equations, see \cite[Appendix~1.3]{marpul} for details.

Assuming that both $v_1$, $v_2$, $\psi$, $\chi$ are also functions of time, then the unsteady Bernoulli law,
\eqnb\label{Bernoulli_real}
p+ \p_t \psi + \frac{1}{2} (v_1^2 +v_2^2 ) = C
\eqne
becomes
\eqnb\label{Bernoulli_complex}
p+ \p_t \re\, \Psi + \frac12 |v|^2 =C.
\eqne
In particular the velocity field $(v_1,v_2)(x,y,t)$ satisfies the 2D incompressible Euler equations \eqref{euler_intro} with pressure function $p$, which follows by taking the gradient of \eqref{Bernoulli_real}. 

If the flow is self-similar, i.e. that
\[
v (z,t) = t^{\mu-1} w \left( \frac{z}{t^\mu } \right),\qquad \Psi (z,t) = t^{2\mu -1} \Phi \left( \frac{z}{t^\mu } \right),\qquad p(z,t) = t^{2\mu-2 } q\left( \frac{z}{t^\mu } \right)
\]
for all $t>0$, where $w(z)=v(z,1)$, $\Phi (z)=\Psi (z,1)$, $q(z)=p(z,1)$. Then $w^*=\Phi'$ and the Bernoulli law \eqref{Bernoulli_complex} takes the form \eqref{Bernoulli_complex_ss},
\[
q(z) + \re \left( (2\mu -1) \Phi(z) - \mu z w^* \right) + \frac{1}2 |w|^2 =C.
\]
In particular, taking ``$\na $'' gives the self-similar form of the Euler equations \eqref{euler_intro},
\eqnb\label{euler_ss}
\na q + (\mu -1 ) w - \mu (z\cdot \nabla )w + (w\cdot \nabla )w =0.
\eqne 

We conclude this section with a technical lemma regarding separation of $\exp (z)$ from $0$.

\begin{lemma}\label{lem-11-aabb}
For every $\ve,r>0$ there exists $c>0$ such that
\begin{align*}
|\ee^{z}-r|\ge c>0\quad\text{for}\quad  z\in \C\setminus(\bigcup_{j\in\Z} B(\ln r + 2j\pi i, \ve)).
\end{align*}
\end{lemma}
\begin{proof} It suffices to show that 
\begin{align*}
|e^{z}-r|\ge c>0\quad \text{ for } z\in B,
\end{align*}
where
\begin{align*}
B\coloneqq  \{z\in \C \ | \ 0\le\mathrm{Im}\,z\le 2\pi\}\setminus(B(\ln r,\ve)\cup B(\ln r + 2\pi i,\ve)).
\end{align*}
We first choose $M>0$ such that
\[
\ee^M \geq 3r/2 \qquad \text{ and } \qquad \ee^{-M} \leq r/2.
\]
Since $|\ee^z-r | \ne 0$ on the compact set $B \cap \{ -M \leq \re \,z \leq M \}$, there exists $c\in (0,r/2)$ such that 
\[
|\ee^z-r | \geq c > 0 \quad \text{ on } B \cap \{ -M \leq \re\, z \leq M \}.
\]
For $\re z \geq M$ we can use the triangle inequality to obtain 
\begin{align*}
|\ee^{z}-r|\ge |\ee^{z}|-r = \ee^{\mathrm{Re}\,z} -r \ge \ee^{M} -r \ge r/2 > c,
\end{align*}
and similarly
\begin{align*}
|\ee^{z}-r|\ge r-|\ee^{z}| = r-\ee^{\mathrm{Re}\,z} \ge r-\ee^{-M} \ge r/2 > c 
\end{align*}
for $\re \, z \leq -M$.  \end{proof}

\section{Velocity field of logarithmic spirals}\label{sec_vel_field}
In this section we prove Theorem~\ref{thm1}, that is we prove a number of properties of self-similar velocity field $v$ (recall \eqref{def_of_v}) with the self-similar profile $w$ given by \eqref{def_of_w},
\[ 
w(z) \coloneqq  \ee^{i\theta } \sum_{k=0}^{M-1} \frac{2ag_{k} }{r(a-i)} \left(  r^{\frac{2a}{a+i} }\ee^{A  (\theta_k -\theta ) } \frac{\ee^{2\pi \jj(r,\theta ,k)A }}{1-\ee^{2\pi A}} \right)^*.
\]

\subsection{The winding number $\jj$}\label{sec_winding_number}

Here we comment on the notion of the winding number $\jj (r,\theta , k )$. 

First we note that $\jj (r,\theta , k )$ decreases by $1$ whenever $z$ crosses the $k$-th spiral from left to right. Indeed, the inequality in \eqref{def_of_jj} becomes an equality when $r=\ee^{a((\theta-2\pi j ) - \theta_k ) }$, i.e. when $r=|Z_k (\theta' )|$ for $\theta'\coloneqq \theta - 2\pi j$, where $j\in \Z$. Since the exponential representation $z=r\ee^{i\theta }$ is invariant with respect to adding $2\pi j$ to $\theta$, the winding number $\jj (r,\theta , k)$ can be thought of as ``the number of the loop of the the spiral $\Sigma_k$''. For example, if $\theta \in [0,2\pi )$ and $r\in [\ee^{a(\theta - \theta_k)}, \ee^{a(\theta +2\pi - \theta_k)})$ then the part of the spiral $\Sigma_k$ given by $\{ Z_k (\theta' ) \colon \theta \in [0,2\pi ) \}$ is the closest segment of $\Sigma_k$ to $z=r\ee^{i\theta }$ in the direction towards the origin.

We now prove some basic properties of the winding number $\jj$.

\begin{proposition}[Properties of the winding number $\jj$]\label{prop_j1}
The winding number $\jj (r,\theta ,k)\in\Z$ satisfies the following properties. 
\begin{enumerate}
\item[(i)] $\jj(r,\theta+2l\pi,k)=\jj(r,\theta,k)+l$ for each $l\in \Z$.
\item[(ii)] $\jj (\cdot , \cdot , k )$ is constant in the set $\{ (r,\theta ) \in (0,\infty )\times \R \colon \ee^{a(\theta-2\pi (j-1) -\theta_k )}\geq r > \ee^{a(\theta-2\pi j -\theta_k )} \}$ for any $j\in \Z$, $k\in \{ 0 ,\ldots , M-1 \}$.  
\item[(iii)] For each $r>0$, $\theta \in \R$, $k\in \{ 0 ,\ldots , M-1 \}$,
\[
-\frac{1}{2\pi}\left(\frac{\ln r}{a}+\theta_{k}-\theta \right)< \jj (r,\theta,k)\le -\frac{1}{2\pi}\left(\frac{\ln r}{a}+\theta_{k}- \theta \right)+1.
\]
\item[(iv)] Given $z=\ee^{a(\theta - \theta_m )}\ee^{i\theta }\in \Sigma_m$ for some $m\in \{ 0,\ldots , M-1 \}$, i.e. that $r=\ee^{a(\theta - \theta_m )}$, let $\jj^{R} (r,\theta , k)$ and $\jj^{L}(r,\theta , k)$ denote the limits of $\jj(r',\theta' , k)$ as  $(r',\theta')\to (r,\theta )$ in a way that $z'=r'e^{i\theta'}\to z$, from the right and left sides of $\Sigma_m$, respectively. Then
\[
\jj^{R} ( r, \theta , k ) = 1_{\theta_{k}<\theta_{m}},\quad \jj^{L} ( r, \theta , k ) = \jj (r,\theta ,k) =1_{\theta_{k}\le\theta_{m}},
\]
\end{enumerate}
\end{proposition}

Note that Proposition~\ref{prop_j1}(i) shows that the velocity field \eqref{def_of_v} and the complex potential \eqref{def_of_potential_fcn_Phi} are well-defined, despite the fact that the domain of $\jj (\cdot, \cdot , k)$ is $(0,\infty ) \times \R$ (instead of $(0,\infty )\times [0,2\pi)$); namely 
\eqnb\label{w_welldef}
w(r\ee^{i\theta } ) = w\left( r\ee^{i(\theta + 2l\pi ) } \right),\quad  \Phi (r\ee^{i\theta } ) = \Phi \left( r\ee^{i(\theta + 2l\pi ) } \right),  \qquad r>0, \theta \in \R, l\in \Z.
\eqne
Moreover, Proposition~\ref{prop_j1}(ii) shows that considering behaviour of $w$ in $\Omega_m$ we can restrict ourselves to $\jj (r,\theta , m)$ constant.

Another way of visualizing the behaviour of $\jj$ is to note that the decrease of $\jj (r,\theta , m)$ by $1$, i.e. the crossing of $\Sigma_m$ from left to right, corresponds to the $m$-th ingredient of the sum \eqref{def_of_w} being multiplied by $\ee^{-2\pi A^*}$. However, since the spirals are not closed contours, the same effect can be achieved by moving with the direction of the spiral (anticlockwise) to the point on the other side of the spiral by, say, keeping $r$ constant and increasing $\theta $ by almost $2\pi$\footnote{During such procedure $\jj (r,\theta , m)$ would remain constant, but all the other $\jj (r, \theta , k)$, for $k\ne m$, would increase by $1$ as during the procedure we must have crossed each of $\Sigma_k$, right to left, see Figure~\ref{fig_spirals} for a sketch. Thus, for $k\ne m$, $\jj(r,\theta ,k)$ has increased by $1$, and so the $k$-th ingredient of the sum \eqref{def_of_w} got multiplied by $\ee^{2\pi A^*}$. Moreover all ingredients got multiplied by $\ee^{-2\pi A^*}$, from the term ``$\ee^{A(\theta_k-\theta )}$'', as $\theta$ increased. Thus the overall effect, i.e. multiplication of the $m$-th term by $\ee^{-2\pi A^*}$, is the same as before.}.\\
\begin{proof}[Proof of Proposition~\ref{prop_j1}.]
Claims (i), (ii) and (iii) follow directly from the definition \eqref{def_of_jj}. 

As for (iv), let us first assume that $k\neq m$. If $(r',\theta')\to (\ee^{a(\theta - \theta_m)},\theta )$, then 
\begin{align*}
-a^{-1}\ln r'+\theta'-\theta_{k} \to \theta_{m}-\theta_{k}\neq 0,
\end{align*}
and so
\[
\jj (r',\theta' ,k) =\min\left\lbrace j\in \Z \colon 2\pi j>-\frac{\ln r'}{a}+\theta'-\theta_{k} \right\rbrace = \min\left\lbrace j\in \Z \colon 2\pi j> \theta_{m}-\theta_{k} \right\rbrace = 1_{\theta_{k}< \theta_{m}}
\]
for $(r',\theta')$ sufficiently close to $(r,\theta )$, as desired. 

If $k=m$ then (iii) gives that $0<\jj (r, \theta ,m) \leq 1$, and so $\jj (r, \theta , m)=1$. If $z'=r'e^{i\theta'}$ approaches $z$ from the left side of $\Sigma_m$ then, since $\ee^{a(\theta-\theta_m )} \ee^{i\theta} \in \Sigma$, we must have 
\[
- \frac{1}{2\pi }\left( \frac{\ln r'}{a} + \theta_m - \theta' \right) \to 0^+,
\]  
and so (iii) gives $\jj^L (r, \theta , m)=1$. Analogously, if $z'=r'e^{i\theta'}$ approaches $z$ from the right then $- \left( a^{-1} {\ln r'} + \theta_m - \theta' \right)/2\pi \to 0^-$, which gives that $\jj^R (r,\theta , m)=0$. 
\end{proof}

\begin{figure}[h]
\centering
 \includegraphics[width=0.5\textwidth]{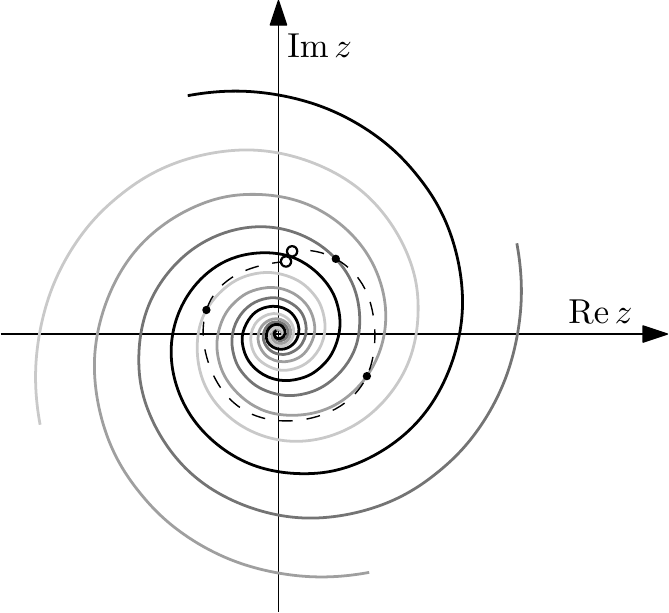}
 \nopagebreak
  \captionof{figure}{A sketch of $\Sigma= \bigcup_{k=0}^{M-1} \Sigma_k$ in the case of Alexander's spirals \eqref{def_of_alex} with $M=4$. Here the two circles at the darkest spiral represent a point on the spiral considered as a limit point from both sides of the spiral, the dashed line illustrates the procedure described below Proposition~\ref{prop_j1}, and the three black dots represent the crossing points with the other spirals that must occur during the procedure.}\label{fig_spirals} 
\end{figure}

\subsection{Proof of Theorem~\ref{thm1}}\label{sec_pf_thm1}
 That $w^*$ is holomorphic in $\Omega_m$ and admits continuous extension to the closure of $\Omega_m$ for each $m$ follows as observed in \eqref{w_simple_form}, by recalling the well-definiteness \eqref{w_welldef} and Proposition~\ref{prop_j1}(ii). In particular this gives that $v$ is irrotational and divergence-free in the regions between the spirals, as these properties are equivalent to the Cauchy-Riemann equations. As for the pointwise estimate \eqref{growth_of_w}, note that Proposition~\ref{prop_j1}(iii) gives
\begin{equation*}
2\pi \jj (r, \theta , k ) > -\frac{\ln r}{a}- (\theta_{k} -\theta )
\end{equation*}
for all $r>0$, $\theta\in\R$, $k\in\Z$. Thus 
\[
\begin{split}
| w(r\ee^{i\theta }) | &\leq  \sum_{k=0}^{M-1} \left|\frac{2ag_{k} }{r(a-i)}   r^{\frac{2a}{a+i} }\ee^{A  (\theta_k -\theta ) } \frac{\ee^{2\pi \jj(r,\theta ,k)A }}{1-\ee^{2\pi A}} \right| \\
&\leq  Cr^{\frac{2a^2}{a^2+1}-1}\sum_{k=0}^{M-1} \ee^{\re\,A  ( (\theta_k -\theta ) + 2\pi \jj(r,\theta ,k) ) } \leq  C r^{\frac{a^2-1}{a^2+1}} \ee^{-\re \,A \frac{\ln r}{a} } = C r 
\end{split}
\]
for every $r>0$, $\theta \in \R$, as required, where we used the fact that $\re  \, A<0$ in the last inequality, and recalled \eqref{def_of_A} in the last step. 

As for claim (ii) of Theorem~\ref{thm1} we recall   that $v$ admits a jump as $z$ crosses $\Sigma_{m}(t)$, for each $m\in \{ 0, \ldots , M-1 \}$, given by the multiplication of the $m$-th element of the sum in \eqref{def_of_w} by $\ee^{-2\pi A^*}$. We now show that this corresponds to the the jump of $v$ in the direction tangent to the spiral. 

Indeed, since $v^R$ and $v^L$ can be expressed using \eqref{def_of_v} by applying the values $\jj^R$ and $\jj^L$ (respectively) of the winding number $\jj$ from Proposition~\ref{prop_j1}(iv), we obtain
\eqnb\label{v_jump}
\begin{split}
v^R &(Z_{m}(\theta,t),t)- v^L(Z_{m}(\theta,t),t) \\
&=  t^{\mu-1} \ee^{i\theta } \sum_{k=0}^{M-1}  \frac{2ag_{k} }{\ee^{a(\theta - \theta_m )} (a-i)} \left( \ee^{\frac{2a^2}{a+i} (\theta - \theta_m) } \ee^{A(\theta - \theta_k )} \frac{ \ee^{2\pi\jj^{R}(|Z_{m}(\theta,1)|, \theta,k) A }-\ee^{2\pi\jj^{L}(|Z_{m}(\theta,1)|, \theta,k) A } }{1-\ee^{2\pi A }}\right)^*\\
&= t^{\mu-1} \ee^{i\theta }  \frac{2ag_{m}\ee^{a (\theta - \theta_m) } }{a-i} = \frac{2a}{a^2+1} g_m\,t^{\mu-1} \ee^{a(\theta - \theta_m )}  \ee^{i\theta } (a+i ),
\end{split}
\eqne
which has the same direction as the tangent to $\Sigma_m (t)$ at $Z_m (\theta,t)$, as
\eqnb\label{Z_theta}
\p_\theta Z_m(\theta,t) = (a+i )  t^\mu \ee^{a(\theta - \theta_m )} \ee^{i\theta }.
\eqne

Since the jump across each spiral occurs in the direction tangent to the spiral, we see that $v\cdot n$ is continuous across each spiral, where $n$ denotes a unit vector orthogonal to the spiral. This together with the decay \eqref{growth_of_w} of $v$ at the origin, and the fact that $\div\, v=0$ pointwise in the regions between the spirals shows that $v$ is weakly divergence free. To be more precise, 
\[
\int_{\R^{2}} v\cdot\nabla\psi = \lim_{\eta\to 0} \int_{B(\eta)^{c}} v\cdot\nabla\psi = \lim_{\eta\to 0}\left(\int_{\Sigma\cap B(\eta )^c } \underbrace{ \left(v^{R} - v^{L}\right) \cdot n}_{=0} \psi -\int_{\p B(\eta ) }\frac{v\cdot z}{|z|} \psi \right)\\
 =  0,
 \]
for every $\psi \in C_0^\infty (\R^2 )$ and every $t>0$, where we used \eqref{growth_of_w} in the last step. 

As for claim (iii) of Theorem~\ref{thm1} we note that, since $w^*$ is holomorphic in $\Omega_m$ for every $m$ we see that $v$ is irrotational in the regions between the spirals. Hence the measure $\mathrm{curl}\,v (t)$ is supported on $\Sigma(t)$ at each time. In order to show the explicit formula \eqref{rotation},
\[
\mathrm{curl}\, v(t) = \sum_{k=0}^{M-1}\gamma(t)\delta_{\Sigma_{k}(t)},\qquad t>0,
\]
where $\gamma(Z_{k}(\theta,t),t)\coloneqq {\p_\theta \Gamma_{k}(t,\theta)} \left|{\p_\theta Z_{k}(\theta,t)} \right|^{-1}$, we note that \eqref{v_jump} and \eqref{Z_theta} give
\[
\begin{split}\p_\theta Z_{k}(\theta,t) \cdot \left( v^{R}(Z_{k}(\theta,t),t) - v^{L}(Z_{k}(\theta,t),t) \right)  &=\mathrm{Re}\left((v^{R}(Z_{k}(\theta,t),t) - v^{L}(Z_{k}(\theta,t),t))(\p_\theta Z_{k}(\theta,t))^*\right) \\
&= 2a\,g_kt^{2\mu -1}\ee^{2a(\theta - \theta_k) } \\
& = \p_\theta \Gamma_k (\theta , t)
\end{split}
\]
for $k\in \{ 0, \ldots , M-1\}$, $\theta\in\R$, $ t>0$, or equivalently
\[
\tau \cdot (v^R-v^L) = \gamma \qquad \text{ on } \Sigma (t),\,t>0,
\]
where 
\[
\tau (Z_{k}(\theta,t),t) \coloneqq \frac{\p_\theta Z_{k}(\theta,t)}{\left| \p_\theta Z_{k}(\theta,t) \right| }
\]
denotes the unit tangent vector to $\Sigma_k(t)$ in the same direction as the orientation of the spiral at $Z(\theta ,t)\in \Sigma_k (t)$. Thus Theorem~\ref{thm1}(iii) follows by writing
\[
\int_{\R^{2}} v^{\bot}\cdot\nabla\psi = \lim_{\eta\to 0} \int_{B(\eta)^{c}} v^{\bot}\cdot\nabla\psi = \lim_{\eta\to 0}\left(\int_{\Sigma\cap B(\eta )^c} (v^{R} - v^{L})\cdot \tau \psi -\int_{\p B(\eta ) }\frac{v^{\bot } \cdot z}{|z|} \psi \right) = \int_{\Sigma}\gamma\,  \psi 
\] 
for any $\psi\in C^{\infty}_{0}(\R^{2})$, where we used \eqref{growth_of_w} in the last step. 

The claim of Theorem~\ref{thm1}(iv) follows  directly from the Bernoulli law, recall the comment below \eqref{Bernoulli_complex}. 

\section{Vortex sheet velocity fields as weak solutions of the Euler equations}\label{sec_euler}
 
In this section we prove Theorem~\ref{thm_main}, that is we show that any vortex sheet velocity (recall Definition~\ref{def_velocity}) is a weak solution of the Euler equations on $\R^2\times (0,\infty )$ if and only if the matching conditions \eqref{velocity_matching}, \eqref{pressure_matching} hold, i.e. that 
\[
n(z,t) \cdot \left( v(z,t)t   - \mu z \right) =0, \quad \text{ and } \quad p \text{ is continuous at }z
\] 
for $z\in \Sigma (t)$, $t>0$.

We fix a divergence-free $\varphi\in C_0^\infty (\R^2 \times (0,\infty ))$ and set
\[
\phi (z,t) \coloneqq  \varphi (zt^\mu ,t).
\]
Using the decay of $w$ at the origin (recall Definition~\ref{def_velocity}(4)) we obtain 
\[\begin{split}
\int_0^\infty &\int_{\R^2}\left( \sum_{i,j=1}^{2} v_i v_j \p_i \varphi_j + v \cdot \varphi_t \right) \\
&= \lim_{\eta \to 0} \int_0^\infty t^{2\mu } \int_{B(\eta )^c} \left[ t^{2\mu -2} \sum_{i,j=1}^{2}w_i w_j t^{-\mu } \p_i \phi_j + t^{\mu -1} w\cdot \phi_t-\mu t^{\mu -1} \sum_{i,j=1}^{2}w_j\p_i \phi_j \frac{z_i}{t}  \right] \d z \, \d t \\
&=\lim_{\eta \to 0}\left(-\int_0^\infty t^{2\mu } \int_{B(\eta )^c} t^{\mu -2} \sum_{i,j=1}^{2}w_i \p_i w_j \phi_j +\int_0^\infty t^{2\mu } \int_{\Sigma_{RL}\cap B(\eta )^c }\sum_{i,j=1}^{2}w_i n_i w_j \phi_j t^{\mu -2}\right.\\
&+ \int_0^\infty t^{3\mu -1} \p_t \left( \int_{B(\eta )^c} w\cdot \phi  \right)\\
&+\mu \int_0^\infty t^{3\mu -2 } \int_{B(\eta )^c} \sum_{i,j=1}^{2} \p_i w_j \phi_j {z_i}+2\mu \int_0^\infty t^{3\mu -2} \int_{B(\eta )^c} \sum_{j=1}^{2}w_j \phi_j   \\
&\left.-\mu \int_0^\infty t^{3\mu-2} \int_{\Sigma_{RL} \cap B(\eta )^c } \sum_{i,j=1}^{2} w_j n_i z_i \phi_j \right),
\end{split}
\]
where, in the second step, we have applied integration by parts in space and noted that $w$ is independent of time. Here we have also used the notation
\[
\int_{\Sigma_{RL} } f \coloneqq   \int_{\Sigma} (f^R-f^L )
\]
with the obvious generalization to the set $\Sigma_{RL} \cap B(\eta )^c $, where $f^R$ and $f^L$ denote the limit values of $f$ at $\Sigma \equiv  \bigcup_{m=0}^{M-1} \Sigma_m$ from the right and left sides, respectively. 
 
Noting that $\phi$ has compact support in time we can integrate the third term on the right-hand side by parts in time to get
\[
\int_0^\infty t^{3\mu -1} \p_t \left( \int_{B(\eta )^c} w\cdot \phi  \right)=-(3\mu -1) \int_0^\infty t^{3\mu -2}  \int_{B(\eta )^c} w\cdot \phi ,
\]
and so we see that, since the Euler equations are satisfied in the classical sense in the region between the spirals (recall Definition~\ref{def_velocity}(3) and \eqref{euler_ss}), the sum of the integrals over $(0,\infty )\times {B(\eta )^c}$ on the right-hand side of the above calculation becomes
\[
\begin{split}
\lim_{\eta \to 0} \int_0^\infty& t^{3\mu -2} \int_{B(\eta )^c} \left( - (w\cdot \na ) w - (\mu -1) w + \mu (z\cdot \nabla )w \right)\cdot \phi \\
& = \lim_{\eta \to 0} \int_0^\infty t^{3\mu -2} \int_{B(\eta )^c} \na q \cdot \phi    = \lim_{\eta \to 0} \int_0^\infty t^{3\mu -2} \int_{\Sigma \cap B(\eta )^c } (q^R-q^L)n \cdot \phi
\end{split}
\]
where $q(z) \coloneqq p(z,1)$, and we used Definition~\ref{def_velocity}(3) in the first equality, the fact that $\mathrm{div}\,\phi=0$ and \eqref{growth_p} in the second equality.

Thus the weak form of the Euler equations holds if and only if 
\[\begin{split}
0&=\int_0^\infty t^{3\mu-2} \int_{\Sigma}  \left( \sum_{i,j=1}^2  n_i  \phi_j \left( (w_i^R w_j^R- w_i^L w_j^L)  -\mu z_i (w_j^R - w_j^L) \right)+ (q^R-q^L)n\cdot \phi \right)\\
&=\int_0^\infty t^{3\mu-2} \int_{\Sigma\cap B(\eta )^c}   \left( \left( n\cdot (w-\mu  z) \right) (w^R-w^L ) + (q^R-q^L)n \right) \cdot \phi 
\end{split}
\]
for all divergence-free $\phi \in C_0^\infty (\R^2\times (0, \infty );\R^2)$, where we used that fact that $n\cdot w^R = n\cdot w^L=: n\cdot w$ (recall Definition~\ref{def_velocity}(2)) in the second line. 
Since the vectors $w^R - w^L$ and $n$ are orthogonal (by the same fact), and $w^R-w^L\ne 0$ (by \eqref{v_jump}), this holds if and only if both matching conditions \eqref{velocity_matching}, \eqref{pressure_matching} hold.
\\

\section{Logarithmic spiral vortex sheets as weak solutions of the Euler equations}\label{sec_log_solutions}

Here we prove Theorem~\ref{thm_alex}, that is we show that in the case of logarithmic spirals, i.e. in the case of self-similar velocity profile \eqref{def_of_w}, the velocity matching \eqref{velocity_matching} and the pressure matching \eqref{pressure_matching} conditions are equivalent to the imaginary and real parts (respectively) of  \eqref{eq-disc2},
 \eqnb\label{eq_disc2_recall}
\frac{1}{\sinh (\pi A) }\sum_{k=0}^{M-1} \mathcal{A}_{mk}g_{k} = \frac{2i}{A} \left(\mu  + \frac{(1-2\mu)(a+i)}{2a}\right)^{*} = -(a^2+1-2\mu +2a\mu i )/2a^2,
\eqne
being valid for all $m\in \{ 0 , \ldots , M-1 \}$, where $\mathcal{A}_{mk}$ is defined in \eqref{def_of_Amk}.

\subsection{The velocity matching}\label{sec_vel_matching}
Here we show the claim for the velocity matching condition \eqref{velocity_matching}.

 Indeed noting  that, by Theorem~\ref{thm1}(ii), on $\Sigma$ we have $n \cdot w = n\cdot (w^R  + w^L )/2$, where $n(z)$ denotes the unit normal vector to $\Sigma $ at $z\in \Sigma$  we consider
 \[
 \left( \frac{w^R(Z_m (\theta )) + w^L (Z_m (\theta ))}{2} \right)^* = \ee^{-i\theta } \sum_{k=0}^{M-1} \frac{ag_k }{a+i}\ee^{a\frac{a-i}{a+i}(\theta - \theta_m )} \ee^{A(\theta_k - \theta )} \frac{\ee^{2\pi 1_{k<m} A}+ \ee^{2\pi 1_{k\leq m} A}}{1-\ee^{2\pi A}}.
 \] 
 Since $n$ has the same direction as 
 \[
  i \p_\theta Z_m (\theta ) = i (a+i) \ee^{a(\theta - \theta_m )} \ee^{i\theta }
 \]
 at any $z=Z_m(\theta )\in \Sigma_m(t)$, we see that \eqref{velocity_matching} holds for $z=Z_m (\theta )$ for all $m\in \{ 0, \ldots , M-1\}$, $\theta \in \R$, if and only if
 \[\begin{split}
 0&=\re \left( i \p_\theta Z_m (\theta ) \left( w (Z_m (\theta )) -\mu Z_m(\theta ) \right)^* \right)\\
  &= \re \left( i(a+i ) \ee^{a(\theta - \theta_m )} \left( \sum_{k=0}^{M-1}    \frac{ ag_k }{a+i}\ee^{a\frac{a-i}{a+i}(\theta - \theta_m )} \ee^{A(\theta_k - \theta )} \frac{\ee^{2\pi 1_{k<m} A}+ \ee^{2\pi 1_{k\leq m} A}}{1-\ee^{2\pi A}} -\mu \ee^{a(\theta - \theta_m )}  \right) \right)\\
 &=  \ee^{2a(\theta - \theta_m )} \im  \left( \frac{a\, \ee^{-\pi A}}{2\sinh (\pi A)} \sum_{k=0}^{M-1}   {g_k }  \ee^{A(\theta_k - \theta_m )} \left( {\ee^{2\pi 1_{k<m} A}+ \ee^{2\pi 1_{k\leq m} A}} \right) +\mu (a+i )   \right)\\
 &=  \ee^{2a(\theta - \theta_m )} \im  \left( \frac{a}{\sinh (\pi A)} \sum_{k=0}^{M-1}  \mathcal{A}_{mk} {g_k }  +\mu (a+i )   \right)
 \end{split}
 \]
 for all $\theta \in \R$ and $m\in\{0,\ldots , M-1\}$, which is equivalent to \eqref{eq_disc2_recall}, as required.

\subsection{The pressure matching condition}\label{sec_p_matching}

Here we show the claim for pressure matching \eqref{pressure_matching}. We recall \eqref{def_of_p} that the self-similar pressure profile $q$ is given by \eqref{Bernoulli_complex_ss},
\eqnb\label{def_of_q_recall}
q(z) \coloneqq - \re \left( (2\mu -1) \Phi(z) - \mu z w^* \right) - \frac{1}2 |w|^2 .
\eqne

We will show below that 
\eqnb\label{to_show_jump_of_w^2}
|w^R (Z_m (\theta ))|^2 - |w^L (Z_m (\theta ))|^2=- \frac{4\ee^{2a(\theta - \theta_m )}a^2 g_m}{a^2+1}   \re \underbrace{\left( \frac{1}{\sinh(\pi A)} \sum_{k=0}^{M-1} \mathcal{A}_{mk} g_k \right)}_{=:K}.
\eqne

Then the claim follows by noting that, by Proposition~\ref{prop_j1}, taking $z=Z_m (\theta )$ in $\Phi (z)$ and $w^* (z)$ (recall \eqref{def_of_potential_fcn_Phi} and \eqref{def_of_w}), makes the ``left and right differences'' of the terms in the sum $\sum_{k=0}^{M-1}$ vanish, except  for $k=m$, and the winding number $\jj (\ee^{a(\theta - \theta_m )}, \theta , k)$ becomes either $1_{k<m}$ or $1_{k\leq m}$, depending on the side of the spiral. To be more precise Proposition~\ref{prop_j1}(iv) gives 
\[
\Phi^R (Z_m (\theta )) - \Phi^L (Z_m (\theta )) = \sum_{k=0}^{M-1} g_k \ee^{Aai(\theta - \theta_m )} \ee^{A(\theta_k - \theta )} \underbrace{ \frac{\ee^{2\pi 1_{k<m} A }-\ee^{2\pi 1_{k\leq m}A } }{1-\ee^{2\pi A}}}_{=\delta_{km}}= g_m \ee^{2a(\theta - \theta_m )}
\]
for all $\theta \in \R$. Thus, since $w^* = \frac{iA}{z} \Phi$, \eqref{def_of_q_recall} gives that 
\[\begin{split}
q^R (Z_m (\theta )) - q^L (Z_m (\theta )) &= g_m \ee^{2a(\theta - \theta_m )}\left(-\re (2\mu -1 - \mu i A ) + \frac{2a^2 }{(a^2+1)}   \re \,K \right),\\
&= \frac{-2a^2 g_m \ee^{2a(\theta - \theta_m )}}{a^2+1 } \left( \frac{2\mu -a^2-1 }{2a^2}-\re \,K \right) .
\end{split}
\]
Since the right-hand side vanishes for all $\theta \in \R$ if and only if $\re\,K = (2\mu-a^2-1)/2a^2$, which is the real part of \eqref{eq-disc2}, as required. \\

It remains to verify \eqref{to_show_jump_of_w^2}. To this end we note that 
\[
1=\ee^{A(\theta - \theta_m )} r^{\frac{2i}{a+i}}
\]
for $z=Z_m (\theta )$, where $r=|z|=\ee^{a(\theta - \theta_m )}$, and so
\[
w^* (Z_m (\theta )) =  \frac{2ra\ee^{-i\theta }}{a+i} \sum_{k=0}^{M-1} g_k  \ee^{A (\theta_k - \theta_m )} \frac{\ee^{2\pi \jj (\ee^{a(\theta - \theta_m )},\theta , k)A}}{1-\ee^{2\pi A}}.
\]
Thus 
\eqnb\label{difference_of_wR_wL}
\begin{split}
&|w^R (Z_m (\theta ))|^2 - |w^L (Z_m (\theta ))|^2 \\
&= \frac{4r^2a^2}{(a^2+1)} \sum_{k,l =0}^{M-1} \underbrace{\frac{g_k g_{l} \ee^{A (\theta_k - \theta_m )}\ee^{A^* (\theta_{l} - \theta_m )}}{(1-\ee^{2\pi A} )(1-\ee^{2\pi A^*})}}_{=:C_{k,l}}\underbrace{\left( \ee^{2\pi (\jj^R ( k) A + \jj^R (l)A^*)}  - \ee^{2\pi ( \jj^L ( k) A + \jj^L (l)A^*)}  \right)}_{=:B_{k,l}}   
\end{split}
\eqne
where we used the shorthand notation $\jj^R (k ) \coloneqq \jj^R (|Z_m(\theta )|,\theta ,k)$, $\jj^L (k ) \coloneqq \jj^L (|Z_m(\theta )|,\theta ,k)$. Note that $\jj^R (k)= 1_{k<m}$ and $\jj^L (k) =\jj (k) = 1_{k\leq m}$ (by Proposition~\ref{prop_j1}), and so $B_{k,l}=0$ if both $k$ and $l$ differ from $m$, which gives that
\[
\sum_{k,l=0}^{M-1} C_{k,l} B_{k,l} = \sum_{\substack{k=0\\ k\ne m }}^{M-1} (C_{k,m} B_{k,m}  + C_{m,k} B_{m,k}) + C_{m,m}B_{m,m} .
\]

Observing that  
\[\begin{split}
B_{m,m} &= 1-\ee^{2\pi (A+A^*) }=\frac12 \left( 1+ \ee^{2\pi A } \right) \left( 1- \ee^{2\pi A^* } \right)+ \frac12 \left( 1+ \ee^{2\pi A^* } \right) \left( 1- \ee^{2\pi A } \right)\\
&= 2\re \left(  \left( 1- \ee^{2\pi A^* } \right)\ee^{\pi A} \cosh (\pi A) \right) ,
\end{split}
\]
and noting that $C_{k,l}= C_{l,k}^*$ and $B_{k,l}= B_{l,k}^*$  with 
\[
B_{k,m} = \left( 1-\ee^{2\pi A^* } \right)\begin{cases}
\ee^{2\pi A }  &\qquad \text{ if }  k<m ,\\
1 &\qquad \text{ if }  k>m,
\end{cases}
\]
we obtain
\[\begin{split}
\sum_{k,l=0}^{M-1} &C_{k,l} B_{k,l} = 2\re \left(  \left( 1- \ee^{2\pi A^* } \right)\ee^{\pi A} C_{m,m}\cosh (\pi A)  + \sum_{\substack{k=0\\ k\ne m }}^{M-1} C_{k,m} B_{k,m} \right)\\ 
&=2\re \left( (1-\ee^{2\pi A^* })\ee^{\pi A} \sum_{k=0}^{M-1} C_{k,m} \begin{cases} \ee^{\pi A } \qquad &k<m \\
\cosh (\pi A ) &k=m \\
\ee^{-\pi A} &k>m 
\end{cases}\hspace{0.2cm} \right) \\
&= -g_m \,\re \left( \frac{1}{\sinh(\pi A)} \sum_{k=0}^{M-1} \mathcal{A}_{mk} g_k \right).
\end{split}
\]
Plugging this into \eqref{difference_of_wR_wL} gives \eqref{to_show_jump_of_w^2}, as required.

\section{Sharpness of Delort's Theorem~\ref{thm_delort}}\label{sec_sharpness}

Here we prove that Theorem~\ref{thm_delort} of Delort \cite{delort} cannot be generalized to $\sigma$-finite measures, or that its generalization produces nonunique solutions; namely we prove Corollary~\ref{cor_sharpness}.\\

We pick any $t_0>0$ as well as any parameters (\ref{def_of_alex}) satisfying (\ref{constraint_for_alex}) with $M=1$. Then, in view of the growth estimate of the vorticity as $|x|\to \infty$ (see \eqref{vort_growth_prandtl} above), $\mathrm{curl}\, v(t_0)\in H_{loc}^{-1}$ (see \cite[Theorem~1.1]{CS}; see also Theorem~\ref{thm1}(iii)). By time-reversibility and Theorem \ref{thm_alex}, we obtain that $u(t,x)\coloneqq v(t_0-t,-x)$ satisfies the Euler equations (\ref{weak_euler}) for $t\in(0,t_0)$. However, we will show below that
\begin{equation}\label{na_spirali}
\int_{B(0,r)}|w|^2=C(a,g,\mu)r^4,
\end{equation}
which, due to the self-similarity (\ref{def_of_v}), implies that
\[
\int_{B(0,1)}|u(t)|^2=(t-t_0)^{4\mu-2}\int_{B(0,(t-t_0)^{-\mu})}|w|^2=C(a,g,\mu)(t-t_0)^{-2}\rightarrow \infty
\]   
as $t\rightarrow t_0$. This shows that the $L_{loc}^\infty(0,\infty;L^2_{loc}(\R^2))$ property of $u$ is violated.

In order to verify \eqref{na_spirali} we use spiral coordinates to parametrize a ball $B(0,r)\subset \R^2$, that is
\[
B(0,r) = \{ P(\theta, \theta' ) \colon  \theta \in (-\infty ,  \theta'+ a^{-1} \ln r ), \quad \theta'\in (0,2\pi ] \},
\]
where 
\[
P(\theta, \theta')=\begin{pmatrix}
P_1 (\theta, \theta' )\\
P_2 (\theta, \theta' )
\end{pmatrix} \coloneqq \ee^{a(\theta-\theta')}\begin{pmatrix}
\cos \theta \\
\sin \theta
\end{pmatrix} =\ee^{a(\theta-\theta')}\ee^{i\theta}.
\]
Note that
\[ \det \nabla P(\theta,\theta')= \begin{vmatrix}
a P_1 -P_2 & -aP_1 \\
aP_2 + P_1 & -a P_2  
\end{vmatrix} =  a(P_1^2+P_2^2)=a\ee^{2a(\theta-\theta')}. 
\]
Moreover, (\ref{def_of_w}) shows that
\begin{equation}\label{po_spiralach}
w(z\ee^{(i+a)\alpha })=\ee^{(i+a)\alpha}w(z) \qquad \text{ for all }z\in \C, \alpha \in \R,
\end{equation}
which gives that $w(P(\theta,\theta'))= w(\ee^{a(\theta-\theta')}\ee^{i\theta})=\ee^{(a+i)(\theta-\theta')} w(\ee^{i\theta' }) = \ee^{(a+i)(\theta-\theta')} w(P(\theta', \theta' )) $, and so 
\begin{eqnarray*}
\int_{B(0,r)}|w|^2 &=&\int_0^{2\pi}\int_{-\infty}^{\theta'+\frac{1}{a}\ln r} |w(P(\theta, \theta'))|^2 a\ee^{2a(\theta-\theta')} \d \theta\, \d \theta'\\
&=&\int_0^{2\pi}a|w(P(\theta',\theta'))|^2\int_{-\infty}^{\theta'+\frac{1}{a}\ln r}\ee^{4a(\theta-\theta')} \d \theta \, \d \theta'\\
&=& 1/4\int_0^{2\pi}|w(P(\theta',\theta'))|^2 \ee^{4a\frac{\ln r}{a}} \d \theta'=\frac{r^4}{4}\int_0^{2\pi}|w(P(\theta',\theta'))|^2 \d \theta',
\end{eqnarray*} 
which gives \eqref{na_spirali}, as required.

\section{The Biot-Savart integral}\label{sec_bs}
In this section we prove Theorem~\ref{thm4}, that we show that if the vorticity is given in terms of circulation $\Gamma_m$ in \eqref{log_spirals} on the $m$-th spiral, $m\in \{ 0, \ldots , M-1 \}$, the compatibility conditions \eqref{cc-1} are satisfied, and $\overline{v}$ is given by the Biot-Savart law \eqref{Biot_Sav}, i.e.
\begin{align*}
\overline{v}(z,t) \coloneqq  \left(\frac{1}{2\pi i}\int_{-\infty}^{\infty}\sum_{k=0}^{M-1}\frac{2at^{2\mu-1}g_{k}\ee^{2a(\theta'-\theta_{k})}\,d\theta'}{z-t^{\mu}\ee^{a(\theta' - \theta_{k})}\ee^{i\theta'}}\right)^{*},
\end{align*}
then $\overline{v}(z,t) = v(z,t)$ for every $t>0$, $z\not \in \Sigma (t) \cup \{0 \}$. Namely we show that the velocity field $\overline{v}$ recovered from the Birkhoff-Rott approach \eqref{spirals_eq}, \eqref{cc-1} is the same as our explicit formula $v$ (given by \eqref{def_of_w}).

First we note that $ \overline{v}(z,t) = t^{\mu -1} \overline{w} (z/t^\mu )$, where
\[
\overline{w}(z) \coloneqq \overline{v}(z,1) =  \left(\frac{1}{2\pi i}\int_{-\infty}^{\infty}\sum_{k=0}^{M-1}\frac{2ag_{k}\ee^{2a(\theta'-\theta_{k})}\,d\theta'}{z-\ee^{a(\theta' - \theta_{k})}\ee^{i\theta'}}\right)^{*}.
\]
Applying the change of variable $\theta' \mapsto \theta'- \theta_k =: \sigma$ for each $k\in \{ 0, \ldots , M-1 \}$, we obtain
\eqnb\label{w_integral_proper}
\ee^{-i\theta}\overline{w}(r\ee^{i\theta})= \left(\frac{1}{2\pi i}\int_{-\infty}^{\infty}\sum_{k=0}^{M-1}\frac{2ag_{k}\ee^{2a\sigma}\,d\sigma}{r-\ee^{(a+i)\sigma+i\Delta_{k}}}\right)^{*},
\eqne
where, for brevity, we used the notation 
\[ \Delta_{k}\coloneqq \theta_{k}-\theta .
\]

We fix $r> 0$ and $\theta \in \R$ such that 
\eqnb\label{none_on_axis}
a(2\pi j+\Delta_k )+\ln r \ne 0
\eqne
for all $j\in \Z$, $k\in \{0, \ldots , M-1\}$. We will prove that 
\eqnb\label{toshow_prop}
\frac{1}{2\pi i}\int_{-\infty}^{\infty}f(\sigma ) \d \sigma = \sum_{k=0}^{M-1} \frac{g_{k}}{r(a+i)}r^{\frac{2a}{a+i}}e^{A\Delta_{k}}\frac{e^{2\pi \jj (r,\theta, k) A}}{1-e^{2\pi A}},
\eqne
where 
\eqnb\label{ass_on_f}
f(\sigma ) \coloneqq  \sum_{k=0}^{M-1} \frac{g_k e^{2a\sigma}}{r-e^{(a+i)\sigma+i\Delta_{k}}}, \quad \sigma \in \R. 
\eqne
Note that \eqref{none_on_axis} is equivalent to $r\ee^{i\theta } \not \in \Sigma \cup \{ 0 \}$ and also equivalent to the denominator in \eqref{ass_on_f} not vanishing for any $\sigma$. Then multiplying \eqref{w_integral_proper} by $\ee^{i\theta }$ gives $\overline{w}(z) = w(z)$ for $z\not \in  \Sigma \cup \{ 0 \}$, as required.

In the remainder of this section we prove \eqref{toshow_prop}. 

We first note that the compatibility conditions \eqref{cc-1} imply that  $f$ has an equivalent form\footnote{Observe that $\frac{\ee^{2a\sigma}}{r-\ee^{(a+i)\sigma+i\Delta_{k}}} = \frac{r^{2}\ee^{-2i\sigma-2i\Delta_{k}}}{r-\ee^{(a+i)\sigma+i\Delta_{k}}} - r\ee^{-2i\sigma-2i\Delta_{k}} - \ee^{(a-i)\sigma-i\Delta_{k}}$, cf. \eqref{intro_alg_id}.},
\eqnb\label{ass_on_f-2}
f(\sigma ) = \sum_{k=0}^{M-1} \frac{g_k r^{2}e^{-2i\sigma-2i\Delta_{k}}}{r-e^{(a+i)\sigma+i\Delta_{k}}}.
\eqne

We consider the poles of $f$, that is we set
\eqnb\label{def_sigmaj}
\begin{split}\sigma_{j}=\sigma_j (r,\theta , k) &\coloneqq  \frac{a\ln r - (2\pi j+\Delta_k)}{1+a^{2}} - i\frac{a(2\pi j+\Delta_k)+\ln r}{1+a^{2}} \\
& = \frac{a-i}{1+a^{2}}\ln r - (2\pi j+\Delta_k)\frac{1+ai}{1+a^{2}}. 
\end{split}
\eqne
for $k\in \{ 0, \ldots , M-1 \}$, $r>0$ and  $j\in \Z$. Observe that the l'H\^{o}pital rule gives that
\[
\begin{aligned}
&\lim_{\sigma\to\sigma_{j}}\frac{e^{2a\sigma}(\sigma-\sigma_{j})}{r-e^{(a+i)\sigma+i\Delta_{k}}} = \lim_{\sigma\to\sigma_{j}} \frac{\partial_{\sigma}(e^{2a\sigma}(\sigma-\sigma_{j}))}{\partial_{\sigma}(r-e^{(a+i)\sigma+i\Delta_{k}})}\\
&\qquad = \lim_{\sigma\to\sigma_{j}} -\frac{2ae^{2a\sigma}(\sigma-\sigma_{j})+e^{2a\sigma}}{(a+i)e^{(a+i)\sigma+i\Delta_{k}}}
= -\frac{e^{2a\sigma_{j}}}{(a+i)e^{(a+i)\sigma_{j}+i\Delta_{k}}}.
\end{aligned}
\]
This shows that each $\sigma_j$ ($j\in \Z$) is a simple pole of $f$ and
\eqnb\label{residua}
\res (f,\sigma_j ) = - \sum_{k=0}^{M-1} \frac{g_k \ee^{2a\sigma_j }}{(a+i) \ee^{(a+i)\sigma_j +i\Delta_k }}.
\eqne
We aim to find the integral on the left-hand side of \eqref{toshow_prop} by contour integration as sketched in Figure~\ref{fig_contours} below. We set 
\eqnb\label{def_of_phi0}
\varphi_0 \coloneqq   \pi +\mathrm{arctan} \left( a\right)  +\varepsilon,
\eqne
where $\varepsilon >0$ is sufficiently small so that $\varphi_0 \in (\pi , 3\pi/2)$ and we also set
\eqnb\label{def_dj}
d_j^2 \coloneqq \left| \frac{\sigma_j + \sigma_{j+1}}2 \right|^2= \frac{\ln^2 r + ((2j+1)\pi + \Delta_k )^2 }{1+a^2}
\eqne
and
\[\begin{split}
\Gamma^j &\coloneqq (-d_j, d_j ) ,\\
\Gamma_{-}^{j}&\coloneqq  \{d_j e^{i\varphi} \colon  \varphi\in[\pi,\varphi_0]\}, \\
\Gamma_{+}^{j}&\coloneqq  \{d_j e^{i\varphi} \colon  \varphi\in[\varphi_0,2\pi]\},
\end{split}
\]
with the anticlockwise orientation (see Figure~\ref{fig_contours}), for $j\geq \jj$. In fact $j\geq \jj$ if and only if $\im\, \sigma_j <0 $, which gives another intuition regarding the meaning of the winding index $\jj= \jj (r,\theta , k)$, introduced in \eqref{def_of_jj}. This implies that, for large $j$, the closed contour 
\[
\Gamma^j \cup \Gamma_{-}^j\cup \Gamma_{+}^j
\]
encloses $\{ \sigma_{\jj} , \sigma_{\jj+1}, \ldots , \sigma_j \}$, see Figure~\ref{fig_contours}. 
 \begin{figure}[h]
\centering
 \includegraphics[width=0.7\textwidth]{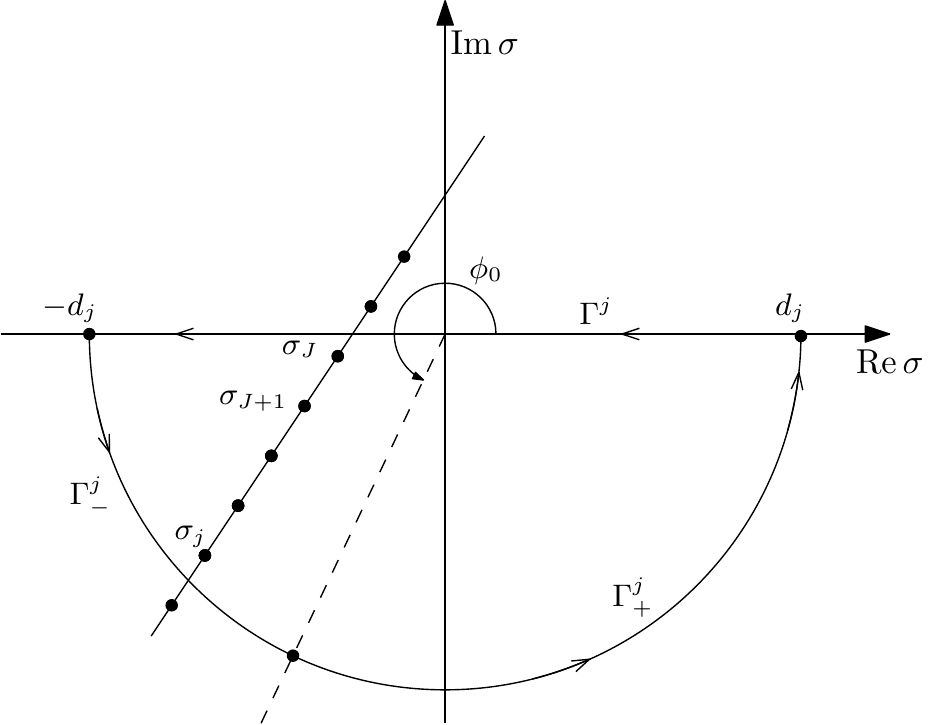}
 \nopagebreak
  \captionof{figure}{A sketch of the contour integration.}\label{fig_contours} 
\end{figure}

If fact, we can say more: the contour $\Gamma_{-}^j\cup \Gamma_{+}^j$ is separated from the poles $\sigma_l$'s by a positive distance. Namely, there exists $j_0\in \Z$ such that
\eqnb\label{cut_in_the_middle}
\left| d_j \ee^{i\varphi } - \sigma_l \right| \geq  \frac{\pi }{4(1+a^2)^{1/2}}
\eqne 
for all $\varphi \in [\pi, 2\pi]$, $l\geq \jj$ and $j\geq j_0$, since for sufficiently large $j$ and every $l\geq \jj$
\[ | d_{j} \ee^{i\varphi} - \sigma_{l}|\ge |d_j - |\sigma_{l}||  \geq \min (d_j - |\sigma_j |,|\sigma_{j+1}|- d_j ) \to \frac{\pi }{2(1+a^2)^{1/2}}\]
as $j\to \infty$, where the last convergence can be verified by a direct calculation using \eqref{def_dj} and the fact that
\[
| \sigma_j |^2 = \frac{\ln^2 r + (2j\pi + \Delta_k )^2 }{1+a^2},
\]
recall \eqref{def_sigmaj}. This shows \eqref{cut_in_the_middle}, see also Figure~\ref{fig_contours}. In Step 1 below we deduce from \eqref{cut_in_the_middle} that 
\eqnb\label{gamma_minus}
\left| \int_{\Gamma_{-}^j} f(\sigma ) \d \sigma \right| \to 0 \qquad \text{ as }j\to \infty.
\eqne
We also show, in Step 2 below, that our choice of $\varphi_0$ \eqref{def_of_phi0} and the equivalent form \eqref{ass_on_f-2} of $f$ give that 
\eqnb\label{gamma_plus}
\left| \int_{\Gamma_{+}^j} f(\sigma ) \d \sigma \right| \to 0 \qquad \text{ as } j\to \infty .
\eqne
The claim of the proposition follows by writing
\[\begin{split}
\frac{1}{2\pi i} \int_{\R} f (\sigma ) \d \sigma  &= \lim_{j\to \infty } \frac{1}{2\pi i} \int_{-d_j}^{d_j} f (\sigma ) \d \sigma \\
&= -\lim_{j\to \infty }\sum_{l=\jj}^j \res \, (f,\sigma_l ) \\
&= \sum_{l\geq \jj}  \sum_{k=0}^{M-1} \frac{g_k \ee^{2a\sigma_l }}{(a+i) \ee^{(a+i)\sigma_l +i\Delta_k }} \\
&=\sum_{k=0}^{M-1}  \frac{g_k }{r(a+i)} \sum_{l\geq \jj}   \ee^{2a\sigma_l }  \\
&= \sum_{k=0}^{M-1}  \frac{g_k }{r(a+i)} r^{\frac{2a}{a+i}} \ee^{A\Delta_k }\sum_{l\geq \jj}   \ee^{2\pi A l}  \\
&=  \sum_{k=0}^{M-1} \frac{g_{k}}{r(a+i)}r^{\frac{2a}{a+i}}e^{A\Delta_{k}}\frac{\ee^{2\pi \jj A}}{1-\ee^{2\pi A }},
\end{split}\]
where we used \eqref{gamma_plus}, \eqref{gamma_minus} in the third line, \eqref{residua} in the fourth line, and noted that $(a+i ) \sigma_l + i\Delta_k =\ln r -2\pi i l $ in the fifth line.

It therefore remains to verify \eqref{gamma_minus}, \eqref{gamma_plus}, which we discuss in Steps 1 and 2 below.\\

\noindent\texttt{Step 1.} We prove \eqref{gamma_minus}. \\

We first deduce from \eqref{cut_in_the_middle} that 
\eqnb\label{denominator_bound}
\left| \ee^{(a+i)\sigma + i \Delta_k}-r \right| \geq c  \qquad \text{ for } \sigma \in \Gamma_{-}^j\cup \Gamma_{+}^j
\eqne
where $c=c(k,r,\theta )$ is another positive constant. Indeed, letting $\varepsilon$ denote the constant on the right-hand side of \eqref{cut_in_the_middle}, we see that  
\[\sigma \in \Gamma_{-}^j\cup \Gamma_{+}^j \subset \left( \bigcup_{l\in\Z} B\left( \sigma_{l}, \varepsilon \right) \right)^c  =\left( \bigcup_{l\in\Z} B\left( \frac{\ln r -i (2\pi l + \Delta_k )}{a+i} , \varepsilon \right) \right)^c . \]
Multiplying $\sigma $ by $(a+i)$ and adding $i\Delta_k$ gives
\[
(a+i)\sigma+i\Delta_k \in\left(\bigcup_{l\in\Z} B(\ln r+2l\pi i,\ve_{0})\right)^{c},
\] 
where $\varepsilon_0 \coloneqq \varepsilon (1+a^2)^{1/2}$. Thus \eqref{denominator_bound} follows from Lemma~\ref{lem-11-aabb}. We parametrize $\Gamma_{-}^j$ by setting 
\eqnb\label{param_gammaj}
\gamma_{j}(\varphi)\coloneqq  d_j \ee^{i\varphi} = d_j (\cos\varphi+i\sin\varphi) 
\eqne
for $\varphi\in[\pi, \varphi_0]$. We observe that 
\[ \re \,\gamma_j (\varphi ) = d_j \cos \varphi \leq d_j \cos \varphi_0 <0  \]
for such $\varphi$ to find that
\[\begin{split}
\left|\int_{\Gamma_{-}^{j}} \frac{\ee^{2a\sigma}\,\d\sigma}{r-\ee^{(a+i)\sigma+i\Delta_k}} \right| & \le  \int_{\pi}^{\varphi_{0}} \frac{\ee^{2a\re\,\gamma_{j}(\varphi)}|\gamma_{j}'(\varphi)|\,\d\varphi}{|r-\ee^{(a+i)\gamma_{j}(\varphi)+i\Delta_k}|}\le  c^{-1}d_j \int_{\pi}^{\varphi_{0}}\ee^{2ad_j \cos\varphi_{0}}\,\d\varphi\\
& \leq  c^{-1}(\varphi_{0}-\pi) d_j \ee^{2ad_j \cos\varphi_{0}} \to 0 
\end{split}
\]
as $j\to \infty$, as required. \\

\noindent\texttt{Step 2.} We prove \eqref{gamma_plus}.\\

Using the equivalent form \eqref{ass_on_f} of $f$ we see that the claim is equivalent to 
\eqnb\label{step1_toshow}
\int_{\Gamma_{+}^{j}} \frac{r^{2}\ee^{-2i\sigma - 2i\Delta_{k}}\,\d\sigma}{r-\ee^{(a+i)\sigma+i\Delta_{k}}}\to 0
\eqne
as $j\to \infty$. We can now consider the parametrization $\gamma_j (\varphi )$, defined in \eqref{param_gammaj} above, for $\varphi \in [ \varphi_0, 2\pi ]$, to see that, for all such $ \varphi$,
\eqnb\label{im_is_nonpositive}
\im\, \gamma_{j}(\varphi) \leq 0,
\eqne
as well as
\eqnb\label{re_away_from_0}
\re \left( (a+i)\gamma_{j}(\varphi)+i\Delta_{k} \right) \equiv d_j(a\cos\varphi-\sin\varphi) \geq d_j c 
\eqne
for some $c>0$, which we now verify.\\

\noindent Case $\varphi\in [\varphi_{0}, 3\pi/2)$. Then $0> \cos\varphi\ge\cos\varphi_{0}$ and $\sin\varphi\le\sin\varphi_{0}<0$, which implies that
\[
 d_j (a\cos\varphi-\sin\varphi)
\ge d_j (a\cos\varphi_{0}-\sin\varphi_{0}) = d_j  \cos\varphi_{0}\,(a-\tan\varphi_{0}) >0, 
\]
where we used the fact that $a = \tan (\pi + \mathrm{arctan}(a)) < \tan \varphi_0 $ (recall \eqref{def_of_phi0}) to obtain the last inequality.\\

\noindent Case $\varphi\in [3\pi/2, 7\pi/4)$. Then $\cos\varphi\ge 0$ and $\sin\varphi\le\sin(7\pi/4)=-\sqrt{2}/2<0$, which gives that 
\[
 d_j (a\cos\varphi-\sin\varphi)  \geq  d_j \sqrt{2}/2>0.
\]

\noindent Case $\varphi\in [7\pi/4,2\pi]$. Then $\cos\varphi\ge \cos(7\pi/4)=\sqrt{2}/2>0$ and $\sin\varphi\le 0$, which implies that
\[ d_j (a\cos\varphi-\sin\varphi) \geq d_j a \sqrt{2}/2>0.\]

Thus \eqref{re_away_from_0} follows by setting
\[
c\coloneqq  \min \left( (a- \tan \varphi_0 ) \cos\varphi_{0}  , \sqrt{2}/2, a\sqrt{2}/2 \right) >0.
\]
The claim \eqref{step1_toshow} now follows by writing
\[
\begin{split}
\left|\int_{\Gamma_{+}^{j}}\frac{r^{2}\ee^{-2i\sigma - 2i\Delta_{k}}\,\d\sigma}{r-\ee^{(a+i)\sigma+i\Delta_{k}}}\right|
&\le r^2\int_{\varphi_{0}}^{2\pi} \frac{|\ee^{-2i\gamma_{j}(\varphi) - 2i\Delta_{k}}||\gamma_{j}'(\varphi)|\,\d\varphi}{|r-\ee^{(a+i)\gamma_{j}(\varphi)+i\Delta_{k}}|}\\
&=r^2d_j \int_{\varphi_{0}}^{2\pi} \frac{\ee^{2\mathrm{Im}\,\gamma_{j}(\varphi)}\,\d\varphi}{|r-\ee^{(a+i)\gamma_{j}(\varphi)+i\Delta_{k}}|}\\
&\leq r^2 d_j \int_{\varphi_{0}}^{2\pi} \frac{\d\varphi}{\ee^{\re ((a+i)\gamma_j (\varphi)+i\Delta_{k} )}-r}\\
& \leq r^2 (2\pi-\varphi_{0})\frac{d_j}{\ee^{d_j c}-r} \to 0
\end{split}
\]
as $j\to \infty$, where we used \eqref{im_is_nonpositive} and the triangle inequality in the second inequality and \eqref{re_away_from_0} in the last inequality.  \\

\begin{remark}\label{rem_average}
We note that for $z\in \Sigma_m$ (for some $m$) the only difference in the above contour integration is the fact that the contour in \eqref{toshow_prop} passes through the pole. (Recall \eqref{def_sigmaj} that $\im\,\sigma_j=0$ if and only if \eqref{none_on_axis} fails.) However, since the pole is simple, such integral equals the average of the integrals along contours that ``avoid'' the pole from each side. This shows that the velocity field $\overline{v}$ on $\Sigma_m(t)$ equals the average of the limit velocities from the two sides of $\Sigma_m(t)$, as mentioned below Theorem~\ref{thm4}.
\end{remark}

\section*{Appendix}
Here we verify that for Alexander's spirals \eqref{def_of_alex} the left-hand side of the complex constraint \eqref{eq-disc2} equals $\coth (\pi A/M)$. Indeed 
\begin{align*}
& \sum_{k=0}^{M-1}\ee^{A(\theta_{k}-\theta_{m})}
\left\{\begin{aligned}
&\ee^{-\pi A}, && \text{if} \ \theta_{k}>\theta_{m},\\
&\cosh(\pi A), && \text{if} \ \theta_{k}=\theta_{m},\\
&\ee^{\pi A}, && \text{if} \ \theta_{k}<\theta_{m},
\end{aligned}\right. \\
&=\sum_{k=0}^{m-1} \ee^{2\pi A (k-m)/M}\ee^{\pi A} + \cosh(\pi A) + \sum_{k=m+1}^{M-1} \ee^{2\pi A(k-m)/M}\ee^{-\pi A} \\
&=\ee^{-2\pi A m/M}\sum_{k=0}^{m-1} \ee^{2\pi A k/M}e^{\pi A} + \cosh(\pi A) + \sum_{k=1}^{M-m-1} \ee^{2\pi A k/M}e^{-\pi A} \\
&=\ee^{-2\pi A m/M}\frac{1-\ee^{2\pi A m/M}}{1-\ee^{2\pi A /M}}\ee^{\pi A} + \cosh(\pi A ) + \left(\frac{1-\ee^{2\pi A (M-m)/M}}{1-\ee^{2\pi A/M}}-1\right)\ee^{-\pi A} \\
&=\frac{\ee^{-2\pi A m/M}-1}{1-\ee^{2\pi A/M}}\ee^{\pi A} + \cosh(\pi A ) + \left(\frac{1-\ee^{2\pi A (M-m)/M}}{1-\ee^{2\pi A/M}}-1\right)\ee^{-\pi A} \\
&=-\frac{\ee^{\pi A}}{1-\ee^{2\pi A /M}} + \frac{\ee^{\pi A}+\ee^{-\pi A}}{2} + \frac{\ee^{-\pi A}}{1-\ee^{2\pi A /M}}-\ee^{-\pi A} \\
&=-\frac{\ee^{\pi A}-\ee^{-\pi A}}{1-\ee^{2\pi A /M}} + \frac{\ee^{\pi A}-\ee^{-\pi A}}{2} = \sinh (\pi A) \frac{1+\ee^{2\pi A /M}}{\ee^{2\pi A /M}-1}
\end{align*}
for each $m\in \{ 1, \ldots , M-2\}$, and the cases $m=0$ or $M-1$ follow analogously.

\section*{Acknowledgments}
T.C. was partially supported by the National Science Centre grant SONATA BIS 7 number UMO-2017/26/E/ST1/00989. WSO was supported in part by the Simons Foundation. P.K. wishes to thank T.C. for his kind hospitality at Institute of Mathematics Polish Academy of Science.

\end{document}